\documentclass{amsart}
\usepackage{amsfonts,amssymb,amscd,amsmath,enumerate,verbatim,newlfont,calc}
\usepackage{graphicx}
\RequirePackage[all]{xy}

\newtheorem{theorem}{Theorem}[section]
\newtheorem{theorem A}{Theorem A}
\newtheorem{theorem B}{Theorem B}
\newtheorem{lemma}[theorem]{Lemma}

\newtheorem{definition}[theorem]{Definition}
\newtheorem{corollary}[theorem]{Corollary}

\newtheorem{proposition}[theorem]{Proposition}
\begin{document}
\authors
\title{The Bogomolov multiplier of Lie superalgebras}
\author[Z. Araghi Rostami]{Zeinab Araghi Rostami}
\author[P. Niroomand]{Peyman Niroomand}
\author[M. Parvizi]{Mohsen Parvizi}
\address{Department of Pure Mathematics\\
Ferdowsi University of Mashhad, Mashhad, Iran}
\email{araghirostami@gmail.com, zeinabaraghirostami@alumni.um.ac.ir}
\address{School of Mathematics and Computer Science\\
Damghan University, Damghan, Iran}
\email{niroomand@du.ac.ir, p$\_$niroomand@yahoo.com}
\address{Department of Pure Mathematics\\
Ferdowsi University of Mashhad, Mashhad, Iran}
\email{parvizi@um.ac.ir}

\address{Department of Mathematics, Ferdowsi University of Mashhad, Mashhad, Iran}

\maketitle
\begin{abstract}
In this paper, we extend the notion of the Bogomolov multiplier and the commutativity preserving extension to Lie superalgebras. Moreover, we compute the Bogomolov multiplier of Heisenberg and real Lie superalgebras of dimension at most $4$. 
\end{abstract}

\section{Introduction and preliminaries}\label{sec1}
In the end of 19th century, Graded Lie algebra has become a topic of interest in physics in the field of '' supersymmetries '' particles related to various statistics. Kac in \cite{7}, introduced the theory of Lie superalgebras, which physicists call them $\mathbb{Z}_2$-graded Lie algebras. Later, similar to Lie's theory, this useful theory has been developed on the connection between Lie superalgebras and Lie supergroups. This theory has made many advances in recent years like many results obtained in representation theory and classification. It should also be said that most of these results are extensions of the well-known facts in Lie algebras. For more information about the Lie superalgebras, see  \cite{5,7,9} and the references given in them.
\\
In this paper, we develop the non abelian commutativity preserving exterior product and the Bogomolov multiplier for Lie superalgebras. This multiplier was first defined for groups by Fedro Bogomolov \cite{4} and it is a group theoretical invariant introduced as an obstruction to a problem in algebraic geometry which is called the rationality problem or Noether's problem. Recently, in \cite{1,2}, we defined this concept for Lie algebras and presented its connection with the Bogomolov multiplier of group by Lazard correspondence. We used a notion of the non abelian exterior square $L\wedge L$ of a finite dimensional Lie algebra $L$ over a field $\mathbb{F}$ to obtain a new description of the Bogomolov multiplier. Using Hopf's formula, we showed that if $0\to R \to F \to L\to 0$ be a free presentation of the finite dimensional Lie algebra over $\mathbb{F}$, then the Bogomolov multiplier is isomorphic to $\frac{R\cap [F,F]}{<K(F)\cap R>}$, where $K(F)=\{[x,y] \ \ | \ \ x,y\in F\}$. Now, It is interesting that the analogous theory of commutativity preserving exterior product can be developed to the theory of Lie (super) theory. The organization of the paper is as follows in sections $2$ and $3$, we introduce the non abelian commutativity preserving exterior product and Hopf's type formula for Lie superalgebras. Finally, in sections $4$ and $5$, we compute the Bogomolov multiplier of Heisenberg Lie superalgebras and real Lie superalgebras of dimension at most $4$.

Throughout this paper, all modules and algebras are defined over an unital commutative ring $\Bbb{K}$. Here, we give some terminology and notations on Lie superalgebras, that are given in \cite{5}.
Let $\Bbb{Z}_2 = \{0, 1\}$ be a field and we put $(-1)^{\bar{0}} = 1$ and $(-1)^{\bar{1}}= -1$.  A $\Bbb{Z}_2$-graded algebra (or superalgebra) $M$ is a direct sum of algebras $M_{\bar{0}}$ and $M_{\bar{1}}$ ($M=M_{\bar{0}} \oplus M_{\bar{1}}$), whose elements are called even and odd, respectively. Non-zero elements of $M_{\bar{0}} \cup M_{\bar{1}}$ are said to be homogeneous. For a homogeneous element $m\in M_{\bar{\alpha}}$ with $\alpha \in \Bbb{Z}_2$, $|m|=\bar{\alpha}$ is the degree of $m$. So whenever we have the notation $|m|$, $m$ will be a homogeneous element. A subalgebra $N$ of $M$ is called $\Bbb{Z}_2$-graded subalgebra (or sub superalgebra), if $N=N_{\bar{0}} \oplus N_{\bar{1}}$ where $N_{\bar{0}}=N\cap M_{\bar{0}}$ and $N_{\bar{1}}=N\cap M_{\bar{1}}$.
\begin{definition}\cite{5}\label{d2.1}
A Lie superalgebra is a superalgebra $M=M_{\bar{0}} \oplus M_{\bar{1}}$ with a multiplication denoted by $[,]$, called super bracket operation, satisfying the following identities:
\renewcommand {\labelenumi}{(\roman{enumi})}
\begin{enumerate}
\item{$[x,y]=-(-1)^{|x||y|}[y,x],$}
\item{$[x,[y,z]]=[[x,y],z]+(-1)^{|x||y|}[y,[x,z]],$}
\item{$[m_{\bar{0}},m_{\bar{0}}]=0$}
\end{enumerate}
for all homogeneous elements $x,y,z\in M$ and $m_{\bar{0}} \in M_{\bar{0}}$.
\end{definition}
Note that the last equation is easily derived from the first equation, in this case, $2$ is invertible in $\Bbb{K}$. The second equation is equivalent to the following graded Jacobi identity
$$(-1)^{|x||z|}[x,[y,z]]+(-1)^{|y||x|}[y,[z,x]]+(-1)^{|z||y|}[z,[x,y]]=0.$$
By using above identities, it can be seen that for a Lie superalgebra $M=M_{\bar{0}} \oplus M_{\bar{1}}$, the even part $M_{\bar{0}}$ is a Lie algebra and the odd part $M_{\bar{1}}$ is a $M_{\bar{0}}$-module. Hence if $M_{\bar{1}}=0$, then $M$ is a Lie algebra and if $M_{\bar{0}}=0$, then $M$ is an abelian Lie superalgebra (i.e. for all $x,y\in M$, $[x,y]=0$). But in general a Lie superalgebra is not a Lie algebra. The sub superalgebra of $L$ is a $\Bbb{Z}_2$-graded vector subspace which is closed under bracket operation. Take $[L, L]$, it is an graded subalgebra of $L$ and is denoted as $L^2$. A $\Bbb{Z}_2$-graded subspace $I$ is a graded ideal of $L$ if $[I, L] \subseteq I$ and for all $x \in L$ the ideal $Z(L) = \{z \in L ; [z, x] = 0\}$ is a graded ideal and it is called the center of $L$. If $I$ is an ideal of $L$, the quotient Lie superalgebra $L/I$ inherits a canonical Lie superalgebra structure such that the natural projection map becomes a homomorphism. The notions of epimorphisms, isomorphisms and auotomorphisms have the obvious meaning. According to the super dimension structure of Lie superalgebras over a field, we say that $L=L_{\bar{0}}\oplus L_{\bar{1}}$ is an $(m , n)$ Lie superalgebra, if $\dim L_{\bar{0}} = m$ and $\dim L_{\bar{1}} = n$. Also throughout $A(m|n)$ denotes an abelian Lie superalgebra with $\dim A = (m | n)$. The descending central sequence of a Lie superalgebra $L$ is defined by $L^{1}=L$ and $L^{c+1}=[L^c,L]$, for all $c\geq 1$. If for some positive integer $c$, $L^{c+1}=0$ and $L^c \neq 0$, then $L$ is called nilpotent with nilpotency class $c$. Also we have $|[m,n]|=|m|+|n|$. 
\begin{definition}\cite{5}\label{d2.2}
Let $M$ and $N$ be two Lie superalgebras. A bilinear map $f: M\to N$ is called a homomorphism of degree $|f|\in \Bbb{Z}_2$ (or Lie super homomorphism), if $f(M_{\bar{\alpha}})\subseteq N_{\bar{\alpha}+|f|}$ and $f([x,y])=[f(x),f(y)]$, for every $x,y\in M$.
\end{definition}
Especially if $|f|=\bar{0}$, then the homomorphism $f$ will be called of even linear map (or even grade).

\begin{definition}\label{d2}
Let $P$ be a Lie algebra and $M$ and $N$ be ideals of $P$. The  exterior product $M\wedge N$ is the Lie superalgebra generated by all symbols $m\wedge n$ subject to the following relations:
\renewcommand {\labelenumi}{(\roman{enumi})}
\begin{enumerate}
\item{$\lambda (m\wedge n) = \lambda m \wedge n = m\wedge \lambda n,$}
\item{$(m+m')\wedge n = m\curlywedge n + m'\wedge n,$ \\ where $m,m'$ have the same graded,}
\item{$m\wedge (n+n') = m\curlywedge n + m\wedge n',$ \\ where $n,n'$ have the same graded,}
\item{$[m,m']\wedge n = m\wedge [m',n]-(-1)^{|m||m'|} m'\wedge [m,n],$}
\item{$m\wedge [n,n'] = (-1)^{|n'|(|m|+|n|)}[n',m]\wedge n-(-1)^{|m||n|} [n,m]\wedge n',$}
\item{$[(m\wedge n),(m'\wedge n')] = -(-1)^{|m||n|}[n,m]\wedge [m',n'],$}
\end{enumerate}
for all $\lambda \in \Bbb{K}$, $m,m'\in M_{\bar{0}}\cup M_{\bar{1}}$ and $n,n'\in N_{\bar{0}}\cup N_{\bar{1}}$. 
\end{definition}
Note that if $M=M_{\bar{0}}$ and $N=N_{\bar{0}}$, then $M\wedge N$ can be considered as a non abelian exterior product of Lie algebras.
\\
\\
A more general structure $M\wedge N$ is given in \cite{5} for arbitrary crossed $P$-modules $M$ and $N$.

\begin{definition}\label{d3}
Let $P$ be a Lie superalgebra and $M$ and $N$ be ideals of $P$. A function $\rho: M\times N \to P$, is called a Lie super exterior pairing, if the following holds:
\renewcommand {\labelenumi}{(\roman{enumi})}
\begin{enumerate}
\item{$\rho({\lambda}m, n) = \rho(m, {\lambda}n) = {\lambda}\rho(m, n),$}
\item{$\rho(m + m', n) = \rho(m, n) + \rho(m', n),$ \\ where $m,m'$ have the same graded,}
\item{$\rho(m, n + n') = \rho(m, n) + \rho(m, n'),$ \\ where $n,n'$ have the same graded,}
\item{$\rho([m,m'],n)=\rho(m,[m',n])-(-1)^{|m||m'|}\rho(m',[m,n]),$}
\item{$\rho(m,[n,n'])=(-1)^{|n'|(|m|+|n|)}\rho([n',m],n)-(-1)^{|m||n|}\rho([n,m],n'),$}
\item{$[\rho(m,n),\rho(m',n')]=-(-1)^{|m||n|}\rho([n,m],[m',n']),$}
\end{enumerate}
for all $\lambda \in \Bbb{K}$, $m,m'\in M_{\bar{0}}\cup M_{\bar{1}}$ and $n,n'\in N_{\bar{0}}\cup N_{\bar{1}}$.
\end{definition}

\section{\bf{The non abelian commutativity preserving exterior product of Lie superalgebras}}\label{sec3}
In this section, we introduce a non abelian commutativity preserving exterior (CP exterior) product of Lie superalgebras, which generalizes the non abelian CP exterior product of Lie algebras in \cite{1}, and then we give some elementary results.

\begin{definition}\label{d3.1}
Let $P$ be a Lie superalgebra and $M$ and $N$ be ideals of $P$. A function $\rho: M\times N \to P$, is called a Lie super ${\tilde{B_0}}$-pairing, if the following holds.
\renewcommand {\labelenumi}{(\roman{enumi})}
\begin{enumerate}
\item{$\rho({\lambda}m, n) = \rho(m, {\lambda}n) = {\lambda}\rho(m, n),$}
\item{$\rho(m + m', n) = \rho(m, n) + \rho(m', n),$ \\ where $m,m'$ have the same graded,}
\item{$\rho(m, n + n') = \rho(m, n) + \rho(m, n'),$ \\ where $n,n'$ have the same graded,}
\item{$\rho([m,m'],n)=\rho(m,[m',n])-(-1)^{|m||m'|}\rho(m',[m,n]),$}
\item{$\rho(m,[n,n'])=(-1)^{|n'|(|m|+|n|)}\rho([n',m],n)-(-1)^{|m||n|}\rho([n,m],n'),$}
\item{$[\rho(m,n),\rho(m',n')]=-(-1)^{|m||n|}\rho([n,m],[m',n']),$}
\item{If $[m,n]+(-1)^{|m'||n'|}[m',n']=0$, then $\rho(m,n)+(-1)^{|m'||n'|}\rho(m',n')=0,$\\
If $[m_{\bar{0}},n_{\bar{0}}]=0$, then $\rho(m_{\bar{0}},n_{\bar{0}})=0$,}
\end{enumerate}
for all $\lambda \in \Bbb{K}$, $m,m'\in M_{\bar{0}}\cup M_{\bar{1}}$, $n,n'\in N_{\bar{0}}\cup N_{\bar{1}}$, $m_{\bar{0}}\in M_{\bar{0}}$ and $n_{\bar{0}}\in N_{\bar{0}}$.
\end{definition}

\begin{definition}\label{d3.2}
A Lie super ${\tilde{B_0}}$-pairing $\rho : M\times N \to L$ is called universal, if for any Lie super ${\tilde{B_0}}$-pairing $\rho' : M\times N \to L'$, there is a unique Lie homomorphism $\theta : L \to L'$ such that $\theta \rho=\rho'$.
\end{definition}
The following definition extends the concept of CP exterior product in \cite{1} to the theory of Lie superalgebras.
\begin{definition}\label{d3.4}
Let $L$ be a Lie algebra and $M$ and $N$ be two ideals of $L$. The CP exterior product $M\curlywedge N$ is the Lie superalgebra generated by all symbols $m\curlywedge n$ subject to the following relations
\renewcommand {\labelenumi}{(\roman{enumi})}
\begin{enumerate}
\item{$\lambda (m\curlywedge n) = \lambda m \curlywedge n = m\curlywedge \lambda n,$}
\item{$(m+m')\curlywedge n = m\curlywedge n + m'\curlywedge n,$ \\ where $m,m'$ have the same graded,}
\item{$m\curlywedge (n+n') = m\curlywedge n + m\curlywedge n',$ \\ where $n,n'$ have the same graded,}
\item{$[m,m']\curlywedge n = m\curlywedge [m',n]-(-1)^{|m||m'|} m'\curlywedge [m,n],$}
\item{$m\curlywedge [n,n'] = (-1)^{|n'|(|m|+|n|)}[n',m]\curlywedge n-(-1)^{|m||n|} [n,m]\curlywedge n',$}
\item{$[(m\curlywedge n),(m'\curlywedge n')] = -(-1)^{|m||n|}[n,m]\curlywedge [m',n'],$}
\item{If $[m,n]+(-1)^{|m'||n'|}[m',n']=0$, then $m\curlywedge n+(-1)^{|m'||n'|} m'\curlywedge n'=0,$\\
If $[m_{\bar{0}},n_{\bar{0}}]=0$, then $m_{\bar{0}} \curlywedge n_{\bar{0}}=0$,}
\end{enumerate}
for all $\lambda \in \Bbb{K}$, $m,m'\in M_{\bar{0}}\cup M_{\bar{1}}$, $n,n'\in N_{\bar{0}}\cup N_{\bar{1}}$, $m_{\bar{0}}\in M_{\bar{0}}$ and $n_{\bar{0}}\in N_{\bar{0}}$.
\end{definition}

In the case $M=N=L=L_{\bar{0}}\oplus L_{\bar{1}}$, we call $L\curlywedge L$ the curly exterior product of $L$ and for any $x,y\in L_{\bar{0}}\cup L_{\bar{1}}$ and $x_{\bar{0}}\in L_{\bar{0}}$, since $[x,y]+(-1)^{|x||y|}[y,x]=0$ and $[x_{\bar{0}},x_{\bar{0}}]=0$, we have
$$x\curlywedge y+(-1)^{|x||y|}y\curlywedge x=0 \ \ \ , \ \ \ x_{\bar{0}}\curlywedge x_{\bar{0}}=0.$$
\textbf{Remark.} Let $m=m_{\bar{0}}+m_{\bar{1}}$ and $n=n_{\bar{0}}+n_{\bar{1}}$ are arbitrary elements of $M$ and $N$ respectively, then according to the definition \ref{d3.4} we have
$$m\curlywedge n=m_{\bar{0}}\curlywedge n_{\bar{0}}+m_{\bar{0}}\curlywedge n_{\bar{1}}+m_{\bar{1}}\curlywedge n_{\bar{0}}+m_{\bar{1}}\curlywedge n_{\bar{1}}.$$
If $M=M_{\bar{0}}$ and $N=N_{\bar{0}}$ then the $M\curlywedge N$ is called the non abeian CP exterior product of Lie algebras which is introduced in \cite{1}.
\begin{lemma}\label{p3.5}
The function $\rho : M\times N \to M\curlywedge N$ given by $(m,n) \longmapsto m\curlywedge n$ is an universal  Lie super ${\tilde{B_0}}$-pairing.
\begin{proof}
The proof is straightforward.
\end{proof}
\end{lemma}
\begin{theorem}\label{t3.6}
Let $L$ be a Lie superalgebra and $M$ and $N$ be two ideals of $L$. Then we have
$$M\curlywedge N\cong \frac{M\wedge N}{{\mathcal{M}_0}(M,N)},$$
where $\mathcal{M}_0(M,N)$ be the graded submodule of $M\wedge N$ generated by the elements
\renewcommand {\labelenumi}{(\roman{enumi})}
\begin{enumerate}
\item{$m\wedge n+(-1)^{|m'||n'|} m'\wedge n',$ where $[m,n]+(-1)^{|m'||n'|}[m',n']=0$,}
\item{$m_{\bar{0}}\wedge n_{\bar{0}},$ where $[m_{\bar{0}},n_{\bar{0}}]=0$,}
\end{enumerate}
with $m,m'\in M_{\bar{0}}\cup M_{\bar{1}}$, $n,n' \in N_{\bar{0}}\cup N_{\bar{1}}$, $m_{\bar{0}}\in M_{\bar{0}}$ and $n_{\bar{0}}\in N_{\bar{0}}$. 

\begin{proof}
By using Definition \ref{d3}, the function $\rho: M\times N \to M\curlywedge N$ given by $(m,n)\longmapsto (m\curlywedge n)$ is a Lie super exterior pairing. So it induces a Lie super homomorphism $\tilde{\rho}: M\wedge N \to M\curlywedge N$, given by $(m\wedge n) \longmapsto m\curlywedge n$, for all $m\in M$ and $n\in N$. Clearly $\mathcal{M}_0(M,N)\subseteq \ker{\tilde{\rho}}$, so we have the Lie super homomorphism ${\rho}^*: {(M\wedge N)}/{\mathcal{M}_0(M,N)}\to M\curlywedge N$ given by $(m\wedge n) + \mathcal{M}_0(M,N)\longmapsto (m\curlywedge n)$. On the other hand, the map ${l}^*: M\curlywedge N \to {(M\wedge N)}/{\mathcal{M}_0(M,N)}$ given by $(m\curlywedge n)\longmapsto (m\wedge n)+{\mathcal{M}_0(M,N)}$ is induced by the Lie super exterior pairing ${l}: {M}\times {N} \to {(M\wedge N)}/{{\mathcal{M}_0(M,N)}}$ given by $(m,n)\longmapsto (m\wedge n)+\mathcal{M}_0(M,N)$. Now it is easy to see that ${{\rho}^*}{l^*}={l^*}{{\rho}^*}=1$. Thus ${l}^*$ is an isomorphism.
\end{proof}
\end{theorem}
It is known that $\kappa: M\times N \to [M,N]$ given by $(m,n)\longmapsto [m,n]$ is a Lie super exterior pairing. So for all $m\in M$ and $n\in N$, it induces a Lie super homomorphism $\tilde{\kappa}: M\wedge N \to [M,N]$, such that $\tilde{\kappa}(m\wedge n)=[m,n]$. Moreover, the kernel of $\tilde{\kappa}$ is denoted by $\mathcal{M}(M,N)$. It can easily seen that $\mathcal{M}_0(M,N) \leq \mathcal{M}(M,N)$, thus there is a Lie super homomorphism ${\kappa}^*: M\wedge N / \mathcal{M}_0(M,N) \to [M,N]$ given by $m\wedge n + \mathcal{M}_0(M,N)\longmapsto [m,n]$, with $\ker {{\kappa}^*}\cong \mathcal{M}(M,N)/\mathcal{M}_0(M,N)$. Similar to the theory of Lie algebras, we denote $\mathcal{M}(M,N)/\mathcal{M}_0(M,N)$ by ${\tilde{B_0}}(M,N)$, and we call it the Bogomolov multiplier of the pair of Lie superalgebras $(M,N)$. Therefore, we have an exact sequence
$$0\to {\tilde{B_0}}(M,N) \to M\curlywedge N \to [M,N] \to 0.$$
In the case $M=N=L$, we denote $\mathcal{M}_0(L,L)$ by $\mathcal{M}_0(L)$.
\\
\\
Similar to Lie algebras, we denote $\mathcal{M}(L)/\mathcal{M}_0(L)$ by ${\tilde{B_0}}(L)$, and we call it the Bogomolov multiplier of the Lie algebra $L$. So we have an exact sequence
$$0\to {\tilde{B_0}}(L) \to L\curlywedge L \to L^2 \to 0.$$
\begin{proposition}\label{p3.7}
Let $L$ be a Lie superalgebra and $M$, $N$ and $K$ be ideals of $L$, such that $K\subseteq M\cap N$. Then there is an isomorphism
$${M}/{K}\curlywedge {N}/{K} \cong {(M\curlywedge N)}/{T},$$
where $T$ be the sub Lie superalgebra of $M\curlywedge N$ generated by the following elements:
\renewcommand {\labelenumi}{(\roman{enumi})}
\begin{enumerate}
\item{$m\curlywedge n+(-1)^{|m'||n'|} m'\curlywedge n',$ where $[m,n]+(-1)^{|m'||n'|}[m',n']\in K$,}
\item{$m_{\bar{0}}\curlywedge n_{\bar{0}},$ where $[m_{\bar{0}},n_{\bar{0}}]\in K$,}
\end{enumerate}
with $m,m'\in M_{\bar{0}}\cup M_{\bar{1}}$, $n,n' \in N_{\bar{0}}\cup N_{\bar{1}}$, $m_{\bar{0}}\in M_{\bar{0}}$ and $n_{\bar{0}}\in N_{\bar{0}}$. 

\begin{proof}
The function $\phi: M\times N\to {M}/{K}\curlywedge {N}/{K}$ given by $(m,n) \to (m+K)\curlywedge (n+K)$ is a well-defined Lie super $\tilde{B_0}$-pairing. Thus there is a Lie super homomorphism ${\phi}^*: M\curlywedge N\to {M}/{K}\curlywedge {N}/{K}$ with $m\curlywedge n\longmapsto (m+K)\curlywedge (n+K)$. Clearly $T\subseteq \ker{{\phi}^*}$, so we have the homomorphism $\psi: {(M\curlywedge N)}/{T}\to {M}/{K}\curlywedge {N}/{K}$ given by $m\curlywedge (n + T)\longmapsto (m+K)\curlywedge (n+K)$. On the other hand, the map ${\varphi}^*: {M}/{K}\curlywedge {N}/{K} \to {(M\curlywedge N)}/{T}$ given by $(m+K)\curlywedge (n+K)\longmapsto (m\curlywedge n)+T$ is induced by the Lie super $\tilde{B_0}$-pairing ${\varphi}: {M}/{K}\times {N}/{K} \to {(M\curlywedge N)}/{T}$ given by $(m+K,n+K)\longmapsto (m\curlywedge n)+T$. One can check that, ${{\varphi}^*}{\psi}={\psi}{{\varphi}^*}=1$. Thus, ${\varphi}^*$ is an isomorphism, and the proof is completed.
\end{proof}
\end{proposition}
Now, we give the behaviour of the CP exterior product respect to a direct sum of Lie superalgebras.
\begin{proposition}\label{p3.8}
Let $M$ and $N$ be two ideals of a Lie superalgebra $L$. Then
$$(M\oplus N)\curlywedge (M\oplus N)\cong M\curlywedge N\oplus N \curlywedge N.$$
\begin{proof}
By using Theorem 4.9 in \cite{13}, we have
\begin{align*}
(M\oplus N)\curlywedge (M\oplus N)&\cong (M\oplus N)\wedge (M\oplus N)+{M_0}(M\oplus N)\\ 
&=(M\wedge M)\oplus (N\wedge N)\oplus {M^{ab}}\otimes {N^{ab}}+{M_0}(M\oplus N),
\end{align*}
since 
$${M_0}(M\oplus N)={M_0}(M)\oplus {M_0}(N)\oplus {M^{ab}}\otimes {N^{ab}},$$
we have
$$(M\oplus N)\curlywedge (M\oplus N)\cong (M\curlywedge M)\oplus (N\curlywedge N).$$

\end{proof}
\end{proposition}
\section{Hopf's type formula for the Bogomolov multiplier of Lie superalgebras}\label{sec4}

Here, we recall from \cite{9,10} the following definitions.
\begin{definition}
The free Lie superalgebra on a $\Bbb{Z}_2$-graded set $X=X_{\bar{0}}\cup X_{\bar{1}}$ is a Lie superalgebra $F(X)$ together with a degree zero map $i:X \to F(X)$ such that  if $M$ is any Lie superalgebra and $j:X \to M$ is a degree zero map, then there is a unique Lie super homomorphism $h:F(X)\to M$ with $j=hoi$.
\end{definition}
The existence of free Lie superalgebra is guaranteed by an analog of Witt's theorem (see \cite{9}, Theorem 6.2.1).

\begin{definition}
Let $L$ be a Lie superalgebra generated by a $\Bbb{Z}_2$-graded set $X=X_{\bar{0}}\cup X_{\bar{1}}$ and $\phi: X\to P$ be a degree zero map, then there is a free Lie superalgebra $F$ and $\psi: F\to L$ expanding $\phi$. Let $R=\ker(\psi)$, the extension $0\to R\to F\to L\to 0$ is named a free presentation of $L$ and is denoted by $(F,\psi)$. 
\end{definition}

According to the well-known Hopf's type formula \cite{12}, we have an isomorphism $\mathcal{M}(P)\cong (R\cap F^2)/[R,F]$. Now we want to introduce the similar formula for the Bogomolov multiplier of a Lie superalgebra $L$.
\\
\\
Let $K(F)$ is used to denote $\{ [x,y] \ |\  x,y\in F \}$. Then

\begin{proposition}\label{p4.1}
Let $L$ be a Lie superalgebra with the free presentation $L\cong {F}/{R}$, then
$${\tilde{B_0}}(L)\cong\dfrac{R\cap F^2}{<K(F)\cap R>}.$$
\begin{proof}
From [\cite{5}, Corrollary 6.5], $L\wedge L \cong {F^2}/[R,F]$ and $L^2 \cong {F^2}/{(R\cap F^2)}$. Moreover $\ker \tilde{\kappa}=\ker(L\wedge L\to L^2)=\mathcal{M}(L)\cong (R\cap {F^2})/[R,F]$ and $\mathcal{M}_0(L)$ can be considered as the Lie sub superalgebra of ${F}/{[R,F]}$ generated by all commutators in ${F}/{[R,F]}$ that belong to $\mathcal{M}(L)$. Thus we have the following Lie super isomorphism for $\mathcal{M}_0(L)$,
$$\mathcal{M}_0(L)\cong <K(\dfrac{F}{[R,F]})\cap \dfrac{R}{[R,F]}> = \dfrac{<K(F)\cap R>+[R,F]}{[R,F]} =\dfrac{<K(F)\cap R>}{[R,F]}.$$
Therefore ${\tilde{B_0}}(L) ={\mathcal{M}(L)}/{\mathcal{M}_0(L)}\cong {R\cap F^2}/{<K(F)\cap R>}$, as required.
\end{proof}
\end{proposition}

\begin{proposition}\label{p4.2}
Let $L$ be a Lie superalgebra with the free presentation $0 \xrightarrow{} R \xrightarrow{}  F \xrightarrow{\pi} L \xrightarrow{} 0$. If $T$ is an ideal in $F$ with $M=T/R$, then the following sequences are exact.
\renewcommand {\labelenumi}{(\roman{enumi})}
\begin{enumerate}
\item{${\tilde{B_0}}(L) \to {\tilde{B_0}}(\dfrac{L}{M})\to \dfrac{M}{<K(L)\cap M>} \to \dfrac{L}{L^2}\to \frac{\frac{L}{M}}{(\frac{L}{M})^2}\to 0$,}
\item{$0\to \dfrac{R\ \cap <K(F)\cap T>}{<K(F)\cap R>} \to {\tilde{B_0}}(L) \to {\tilde{B_0}}(\dfrac{L}{M})\to \dfrac{M\cap L^2}{<K(L)\cap M>} \to 0$.}
\end{enumerate}

\begin{proof}
\textbf{(i)}\ The inclusion maps $R\cap {F^2}\xrightarrow{f} T\cap {F^2},\ \ $ $T\cap {F^2}\xrightarrow{g} T,\ \ $ $T\xrightarrow{h} F$ and $F\xrightarrow{k} F$ induce the sequence of homomorphisms \\
$\displaystyle{\frac{R\cap {F^2}}{<K(F)\cap R>} \xrightarrow{f^*} \frac{T\cap {F^2}}{<K(F)\cap T>} \xrightarrow{g^*} \frac{T}{<K(F)\cap T>+R} \xrightarrow{h^*} \frac{F}{R+{F^2}} \xrightarrow{k^*}}$\\ $\displaystyle{\frac{F}{T+{F^2}} \to 0}$. Note that
$\displaystyle{\frac{T}{<K(F)\cap T>+R}\cong \frac{M}{<K(L)\cap M>}}$, $\displaystyle{\frac{F}{R+{F^2}}\cong \dfrac{L}{L^2}}$ and\\ $\displaystyle{\frac{F}{T+{F^2}}\cong \frac{{L}/{M}}{({L}/{M})^2}}$.
Now by using Proposition \ref{p4.1}, we have \\
$\displaystyle{{\tilde{B_0}}(L)\cong \frac{R\cap{F^2}}{<K(F)\cap R>}}$ and $\displaystyle{{\tilde{B_0}}(\dfrac{L}{M})\cong\frac{T\cap{F^2}}{<K(F)\cap T>}}$. Moreover,\\ $\displaystyle{\text{Im}{f^*}=\text{Ker}{g^*}=}$ $\displaystyle{\frac{R\cap{F^2}}{<K(F)\cap T>}},\ \ $ $\displaystyle{\text{Im}{g^*}=\text{Ker}{h^*}=\frac{T\cap{F^2}}{<K(F)\cap T>+R}}$, \\ $\displaystyle{\text{Im}{h^*}=\text{Ker}{k^*}=\frac{T}{R+{F^2}}}$, and $K^*$ is a Lie super epimorphism. Thus the above sequence is exact.
\\
\\
\textbf{(ii)}\ The inclusion maps $$R\ \cap <K(F)\ \cap \ T>\xrightarrow{f} R\ \cap {F^2},\ \ R\ \cap {F^2}\xrightarrow{g} T\cap {F^2}$$ and the map $T\cap {F^2}\xrightarrow{h} (T\cap {F^2})+R$ induce the sequence of Lie super homomorphisms \\
$\displaystyle{0\to \frac{R\  \cap <K(F)\cap T>}{<K(F)\cap R>}\xrightarrow{f^*}\frac{R\cap {F^2}}{<K(F)\cap R>} \xrightarrow{g^*} \frac{T\cap {F^2}}{<K(F)\cap T>} \xrightarrow{h^*}}$\\ $\displaystyle{\frac{(T\cap {F^2})+R}{<K(F)\cap T>+R} \to 0}$. It is straightforward to verify that \\
$\displaystyle{<K(L)\cap M>=\frac{<K(F)\cap T>+R}{R}}$ and $\displaystyle{M\cap {L^2}=\frac{T}{R}\cap \frac{F^2+R}{R}=\frac{(T\cap F^2)+R}{R}}$. \\
Therefore, we have
$$\displaystyle{\frac{M\cap {L^2}}{<K(L) \cap M>}=\frac{({(T\cap R^2)+R})/{R}}{({<K(F)\cap T>+R})/{R}}\cong \frac{(T\cap {F^2})+R}{<K(F) \cap T>+R}}.$$
Now by using Proposition \ref{p4.1}, we have\\
$\displaystyle{{\tilde{B_0}}(L)\cong \frac{R\cap{F^2}}{<K(F)\cap R>}},\ \ $ $\displaystyle{{\tilde{B_0}}(\dfrac{L}{M})\cong\frac{T\cap{F^2}}{<K(F)\cap T>}},\ \ $ and \\ $\displaystyle{\text{Im}{f^*}=\text{Ker}{g^*}=\frac{R\ \cap<K(F)\cap T>}{<K(F)\cap R>}},\ \ $ $\displaystyle{\text{Im}{g^*}=\text{Ker}{h^*}=\frac{R\cap{F^2}}{<K(F)\cap T>}}$. \\ Moreover, $h^*$ is a Lie super epimorphism. Thus the above sequence is exact.
\end{proof}
\end{proposition}
We know for Lie algebras, the Schur multiplier is an universal object of central extensions. Recently, parallel to the classical theory of central extensions, we in \cite{1,2} developed a version of extensions that preserve commutativity and showed that the Bogomolov multiplier is the universal object parametrizing such extensions for a given Lie algebra. Here, we want to introduce a similar notion for Lie superalgebras.

\begin{definition}\label{d4.4}
Let $L$, $M$ and $C$ be Lie superalgebras. An exact sequence of Lie superalgebras $0 \xrightarrow{} M \xrightarrow{\chi}  C \xrightarrow{\pi} L \xrightarrow{} 0$ is called a commutativity preserving extension (CP extension) of $M$ by $L$, if commuting pairs of elements of $L$ have commuting lifts in $C$. A special type of CP extension with the central kernel is named a central CP extension.
\end{definition}
\begin{proposition}\label{p4.5}
Let $e : 0 \xrightarrow{} M \xrightarrow{\chi}  C \xrightarrow{\pi} L \xrightarrow{} 0$ be a central extension. Then $e$ is a CP extension if and only if $\chi (M) \cap K(C)=0$.
\begin{proof}
Suppose that $e$ is a CP central extension. Let $[c_1,c_2] \in \chi (M) \cap K(C)$, then there is a commuting lift $(c'_1,c'_2)\in C\times C$ of the commuting pair $(\pi(c_1),\pi(c_2))$, such that $\pi(c'_1)=\pi(c_1)$ and $\pi(c'_2)=\pi(c_2)$. So for some $a,b \in \chi (M)$, we have $c'_1={c_1}+a$ , $c'_2={c_2}+b$. Therefore $0=[c'_1,c'_2]=[{c_1}+a,{c_2}+b]=[c_1,c_2]$. Hence $\chi (M) \cap K(C)=0$. Conversely suppose that $\chi (M) \cap K(C)=0$. Choose $x,y\in P$ with $[x,y]=0$, we have $x=\pi(c_1)$ and $y=\pi(c_2)$ for some $c_1 , c_2 \in C$. Therefore $\pi ([c_1,c_2])=0$. Hence $[c_1,c_2] \in \chi (M) \cap K(C)=0$, so $[c_1,c_2]=0$. Thus the central extension $e$ is a CP extension.
\end{proof}
\end{proposition}
\begin{definition}\label{d4.6}
An abelian ideal $M$ of a Lie superalgebra $L$ is called  a CP Lie sub superalgebra of $L$ if the extension $0\to M \to L \to \dfrac{L}{M} \to 0$ is a CP extension.
\end{definition}
Also, by using Proposition \ref{p4.5} an abelian ideal $M$ of a Lie superalgebra $L$ is a CP Lie sub superalgebra of $L$ if $M\cap K(L)=0$.

\begin{proposition}\label{p.4.4}
Let $L$ be a Lie superalgebra with the free presentation $0 \xrightarrow{} R \xrightarrow{}  F \xrightarrow{\pi} L \xrightarrow{} 0$. If $T$ is an ideal in $F$ with $N=T/R$, then the following are equivalent.
\begin{enumerate}
\item{$N$ is a CP sub Lie superalgebra of $L$},
\item{The canonical map $\psi: \mathcal{M}_0(L) \to \mathcal{M}_0(L/N)$ is surjective},
\item{The canonical map $\varphi: L\curlywedge L \to L/N \curlywedge L/N$ is an isomorphism}.
\end{enumerate}
\end{proposition}

\begin{proof}
We know there is a Lie superalgebra homomorphism
$$\psi: \frac{<K(F)\cap R>}{[R,F]} \to \frac{<K(F)\cap T>}{[T,F]}$$
$$x+[R,F] \mapsto x+[T,F].$$
Now, as $N$ is CP sub Lie superalgebra of $L$, then by using the Proposition \ref{p4.5} and Definition \ref{d4.6}, $N \cap K(L)=0$. So, $T\cap K(F)\subseteq R$. Hence, $<K(F)\cap R>=<K(F)\cap T>$. Thus $\text{Im} \psi=<K(F)\cap T>/[T,F]$ and $\psi$ is surjective. On the other hand, if $\psi$ be surjective, then $<K(F)\cap T>=<K(F)\cap R>$. So, $K(F)\cap T \subseteq R$. Thus $N\cap K(L)=0$ and $N$ is CP Lie sub algebra of $L$. Therefore $(i)$ and $(ii)$ are equivalent. \\
Now, let $N$ be a CP sub Lie algebra of $L$, then $N\cap K(L)=0$. So, $T\cap K(F)\subseteq R$. By using Proposition \ref{p4.1}, $L\curlywedge L\cong {F^2}/<K(F)\cap R>$ and for all $x\in {F^2}$, we have
$$\varphi: \frac{F^2}{<K(F)\cap R>}\to \frac{F^2}{<K(F)\cap T>}$$
$$x+<K(F)\cap R>\mapsto x+<K(F)\cap T>.$$
Also, 
$$\text{Ker} \varphi=\{x+<K(F)\cap R> \ | \ x\in <K(F)\cap T> \}=\frac{<K(F)\cap T>}{<K(F)\cap R>}.$$
Since $T\cap K(F)\subseteq R$, then $T\cap K(F)>\subseteq <R\cap K(F)>$. So $\text{Ker} {\varphi}=0$ and $\varphi$ is injective.
Also, $\text{Im}{\varphi}={F^2}/<K(F)\cap S>$. Thus, $\varphi$ is surjective. Hence $\varphi$ is isomorphism. On the other hand, if $\varphi$ be an isomorphism, then $<K(F)\cap T>\subseteq <K(F)\cap R>$ and $K(F)\cap T\subseteq R$. Thus $N\cap K(L)=0$. Hence $N$ is a CP sub Lie superalgebra of $L$.
\end{proof}

Now we want to obtain an explicit formula for the Bogomolov multiplier of a direct product of two Lie superalgebras. The following lemma gives a free presentation for ${L_1}\oplus{L_2}$, in terms of the given free presentation for $L_1$ and $L_2$. It will be used in the rest.
\begin{lemma}[\cite{12}, Lemma 5.5]\label{l4.7}
Let $F_1$ and $F_2$ be two free Lie superalgebras freely generated by $X$ and $Y$, respectively and $F=F_1 \divideontimes F_2$ be the free product of $F_1$ and $F_2$. Then 
$0\to R\to F\xrightarrow{\delta}\to L_1\oplus L_2$ is a free presentation for $L_1\oplus L_2$ in which $R=R_1+R_2+[F_1,F_2]$.
\end{lemma}

\begin{proposition}\label{p4.8}
Let $P_1$ and $P_2$ be Lie superalgebras. Then
$${\tilde{B_0}}(L_1\oplus L_2) \cong {\tilde{B_0}}(L_1) \oplus {\tilde{B_0}}(L_2).$$
\begin{proof}
$L_1\oplus L_2$ is a Lie superalgebra with even part $(L_1\oplus L_2)_{\bar{0}}={(L_1)}_{\bar{0}}+{(L_2)}_{\bar{0}}$ and odd part $(L_1\oplus L_2)_{\bar{1}}={(L_1)}_{\bar{1}}+{(L_2)}_{\bar{1}}$. Now by using Lemma \ref{l4.7}, we have
$${\tilde{B_0}}(L_1\oplus L_2) = \dfrac{(R_1+ R_2+[F_2,F_1])\cap (F_1\ast F_2)^2}{<K(F_1\ast F_2)\cap (R_1+ R_2+[F_1,F_2])>}.$$
\\
Let $F = F_1\ast F_2$, then the Lie super epimorphism $F\to F_1\times F_2$ induces the following Lie super epimorphism
$$\alpha : \dfrac{R\cap F^2}{<K(F)\cap R>}\to  \dfrac{R_1\cap {F_1}^2}{<K(F_1)\cap R_1>} \oplus \dfrac{R_2\cap {F_2}^2}{<K(F_2)\cap R_2>}$$
$$x+<K(F)\cap R>\longmapsto (x_1 + <K(F_1)\cap R_1>\ ,\ x_2 + <K(F_2)\cap R_2>)$$
where $x=x_1 + x_2$, $x_1\in R_1\cap {F_1}^2$ and $x_2\in R_2\cap {F_2}^2$.\\
On the other hand,
$$\beta : \dfrac{R_1\cap {F_1}^2}{<K(F_1)\cap R_1>} \oplus \dfrac{R_2\cap {F_2}^2}{<K(F_2)\cap R_2>} \to  \dfrac{R\cap F^2}{<K(F)\cap R>}$$
given by
$$ (x_1 + <K(F_1)\cap R_1>\ ,\ x_2 + <K(F_2)\cap R_2>) \longmapsto x+<K(F)\cap R>$$
is a well-defined Lie super homomorphism. It is easy to check that $\beta$ is a right inverse to $\alpha$, so $\alpha$ is a Lie super epimorphism. Now, we show that $\alpha$ is a Lie super monomorphism. Let $x+<K(F)\cap R>\ \in \ker \alpha$, such that, $x=t_1 + t_2$. So we have $t_1\in <K(F_1)\cap R_1>$ and $t_2\in <K(F_2)\cap R_2>$. Since $t_1 , t_2 \in <R\cap K(F)>$, then $x\in <K(F)\cap R>$. Thus $\alpha$ is a Lie super monomorphism.
\end{proof}
\end{proposition}

Heisenberg Lie superalgebras obtain an isomorphic image of the canonical commutation relations in quantum mechanics. For this reason, these algebras are of interest to physicists. On the other hand, in geometry, just like the Heisenberg group, the study of nilmanifolds starts with the study of some of the simplest nilpotent Lie groups. 
\\
\\
In the next section, we are interesting to compute the Bogomolov multiplier for Heisenberg Lie superalgebras.
\section{Computing the Bogomolov multiplier of Heisenberg Lie superalgebras}\label{sec5}
Heisenberg Lie superalgebra is a Lie superalgebra $L=L_{\bar{0}}\oplus L_{\bar{1}}$ such that its first derived ideal is equal to its $1$-dimensional homogeneous center, i.e. $Z(L)=L^2$ and $\dim L^2=1$.
\\
\\
According to the  this definition, there are two cases.\\
(1). \ if the $1$-dimensional center is even, then $L_0$ is the well-known Heisenberg Lie algebra,  \\
(2). \ if the $1$-dimensional center is odd then $L_0$ is an abelian Lie algebra.

\begin{definition}\cite{14}\label{d5.1}
A special Heisenberg Lie superalgebra is a Heisenberg Lie superalgebra with even center.
\end{definition}

\begin{theorem}\cite{10}\label{t5.2}
Every special Heisenberg Lie superalgebra has dimension $(2m+1|n)$, and it is isomorphic to $H_{(m,n)}=H_{\bar{0}}\oplus H_{\bar{1}}$, where
$$H_{\bar{0}}=<x_1,x_2,...,x_m,x_{m+1},...,x_{2m},z \ \ | \ \ [x_i,x_{m+i}]=z \ \ ; \ \ i=1,...,m>$$
and
$$H_{\bar{1}}=<y_1,y_2,...,y_{n}\ \ | \ \ [y_j,y_{j}]=z \ \ ; \ \ j=1,...,n>.$$
\end{theorem}

\begin{theorem}\cite{10}\label{t5.3}
Every Heisenberg Lie superalgebra with odd center has dimension $(m|m+1)$, and it is isomorphic to $H_m=H_{\bar{0}}\oplus H_{\bar{1}}$, where
$$H_{m}=<x_1,x_2,...,x_{m}, y_1,y_2,...,y_{m}, z \ \ | \ \ [x_j,y_{j}]=z \ \ ; \ \ j=1,...,m>.$$

\end{theorem}

\begin{theorem}\label{t5.4}
Every Heisenberg Lie superalgebra with odd center has trivial Bogomolov multiplier.
\begin{proof}
By using Theorem \ref{t5.3}, we have
$$H_{m}=<x_1,x_2,...,x_{m}, y_1,y_2,...,y_{m}, z \ \ | \ \ [x_j,y_{j}]=z \ \ ; \ \ j=1,...,m>.$$
According to the Definition \ref{d2}, we have
$$z\wedge z=z\wedge [x_j,y_j]=(-1)^{|y_j|(|z|+|x_j|)}([y_j,z]\wedge x_j)=0\wedge x_j=0.$$
Thus
\begin{align*}
H_m\wedge H_m=&<x_i\wedge x_j  \ \ ; \ \ i,j=1,...,m \ , \ i\neq j >\\
&\cup <x_i\wedge y_j , x_i\wedge z , y_i\wedge y_j , y_i\wedge z \ \ ; \ \ i,j=1,...,m>.
\end{align*}
Now for all $w\in \mathcal{M}(H_m)\leq H_m\wedge H_m$, there exist ${\alpha}_{i,j}, {\beta}_{i,j}, {\beta}'_{j}, {\gamma}_{i}, {\eta}_{i,j}, {\delta}_{i}\in \Bbb{K} \  (i,j=1,...,m)$ such that

\begin{align*}
w&=\sum_{\substack{i,j=1 \\ j\neq i}}^{m}{{\alpha}_{i,j}}(x_i \wedge x_j)+\sum_{\substack{i,j=1 \\ j\neq i}}^{m}{{\beta}_{i,j}}(x_i \wedge y_j)+{\sum^{m}_{j=1}}{{\beta}'_{j}}(x_j \wedge y_j)\\
&+{\sum^{m}_{i=1}}{{\gamma}_{i}}(x_i \wedge z)+{\sum^{m}_{i,j=1}}{{\eta}_{i,j}}(y_i \wedge y_j)+{\sum^{m}_{i=1}}{{\delta}_{i}}(y_i \wedge z).
\end{align*}

Let $\tilde{\kappa}: H_m\wedge H_m \to [H_m,H_m]$ be given by $x\wedge y \mapsto [x,y]$. Since $\tilde{\kappa}(w)=0$, we have

\begin{align*}
w&=\sum_{\substack{i,j=1 \\ j\neq i}}^{m}{{\alpha}_{i,j}}[x_i,x_j]+\sum_{\substack{i,j=1 \\ j\neq i}}^{m}{{\beta}_{i,j}}[x_i,y_j]+{\sum^{m}_{j=1}}{{\beta}'_{j}}[x_j ,y_j]\\
&+{\sum^{m}_{i=1}}{{\gamma}_{i}}[x_i,z]+{\sum^{m}_{i,j=1}}{{\eta}_{i,j}}[y_i, y_j]+{\sum^{m}_{i=1}}{{\delta}_{i}}[y_i,z]=0.
\end{align*}

So ${\sum^{m}_{j=1}}{{\beta}'_{j}} z=0$. Hence, ${\sum^{m}_{j=1}}{{\beta}'_{j}}=0$ and ${\beta}'_m=-{\sum^{m-1}_{j=1}}{{\beta}'_{j}}$. Thus
\begin{align*}
w&=\sum_{\substack{i,j=1 \\ j\neq i}}^{m}{{\alpha}_{i,j}}(x_i \wedge x_j)+\sum_{\substack{i,j=1 \\ j\neq i}}^{m}{{\beta}_{i,j}}(x_i \wedge y_j)+{\sum^{m-1}_{j=1}}{{\beta}'_{j}}(x_j\wedge y_j-x_m\wedge y_m)\\
&+{\sum^{m}_{i=1}}{{\gamma}_{i}}(x_i \wedge z)+{\sum^{m}_{i,j=1}}{{\eta}_{i,j}}(y_i \wedge y_j)+{\sum^{m}_{i=1}}{{\delta}_{i}}(y_i \wedge z).
\end{align*}
On the other hand, we have
\[{x_i}\wedge z={x_i}\wedge z+(-1)^{|x_i||x_i|}(x_i\wedge x_i) \ \ ; \ \ [x_i,z]+(-1)^{|x_i||x_i|}[x_i,x_i]=0.\]
Therefore $x_i\wedge z\in \mathcal{M}_0(H_m)$.
Similarly $y_i\wedge z , y_i\wedge y_j$ and for $i\neq j$, $x_i\wedge y_j$, belong to $\mathcal{M}_0(H_m)$. 

Also we have
$$x_j\wedge y_j-x_m\wedge y_m=(-x_m)\wedge y_m+x_j\wedge y_j=(-x_m)\wedge y_m+(-1)^{|x_j||y_j|}x_j\wedge y_j,$$
and
$$[-x_m,y_m]+(-1)^{|x_j||y_j|}[x_j,y_j]=-[x_m,y_m]+[x_j,y_j]=-z+z=0.$$
So, for all $i,j=1...m-1$, $(x_i\wedge y_j-x_m\wedge y_m)\in \mathcal{M}_0(H_m)$. Thus $\mathcal{M}(H_m)\subseteq \mathcal{M}_0(H_m)$ and $\tilde{B_0}(H_m)=0$.
\end{proof}
\end{theorem}

\begin{theorem}\label{t5.5}
Every special Heisenberg Lie superalgebra has trivial Bogomolov multiplier.
\begin{proof}
By using Theorem \ref{t5.2}, $H_{(m,n)}=H_{\bar{0}}\oplus H_{\bar{1}}$ and

$$H_{\bar{0}}=<x_1,x_2,...,x_m,x_{m+1},...,x_{2m},z \ \ | \ \ [x_i,x_{m+i}]=z \ \ ; \ \ i=1,...,m>$$
and
$$H_{\bar{1}}=<y_1,y_2,...,y_{n}\ \ | \ \ [y_j,y_{j}]=z \ \ ; \ \ j=1,...,n>.$$

We can see that
\begin{align*}
H_{(m,n)}\wedge H_{(m,n)}&=<x_i\wedge x_j \ \ ;\ \ i=1,...,2m \ , \  j=1,...,n \ , \ i\neq j>\\
& \cup <x_i\wedge z , x_i\wedge y_j , y_i\wedge y_j   \ \ ; \ \ i=1,...,2m \ , \ j=1,...,n>.
\end{align*}
By using Definition \ref{d2}, we have
$$x_i\wedge z=x_i\wedge [y_j,y_j]=(-1)^{|y_j|(|x_i||y_j|)}([y_j,x_i]\wedge y_j)-(-1)^{|x_i||y_j|}([y_j,x_i]\wedge y_j)=0,$$
and $y_i\wedge z=y_i\wedge [x_i,x_m+i]=0$. So
\begin{align*}
H_{(m,n)}\wedge H_{(m,n)}&=<x_i\wedge x_j \ \ ; \ \ i=1,...,2m \  , \ j=1,...,n \ , \ i\neq j>\\
&\cup <x_i\wedge y_j , y_i\wedge y_j   \ \ ; \ \ i=1,...,2m \ , \  j=1...n>.
\end{align*}
Now for all $w\in \mathcal{M}(H_{(m,n)})\subseteq H_{(m,n)}\wedge H_{(m,n)}$, there exist ${\alpha}_{i,j}, {\beta}_{i,j} , {\alpha}'_{i} , {\gamma}_{j}\in \Bbb{K}$ with $(i=1,...,2m \ , \ j=1,...,n)$ such that

\begin{align*}
w&={\sum^{2m}_{i=1}}{{\alpha}'_{i}}(x_i \wedge x_{m+i})+{\sum^{2m}_{i=1}}{\sum_{\substack{j=1 \\ j\neq m+i}}^{n}{{\alpha}_{i,j}}(x_i\wedge x_j)}+{\sum^{2m}_{i=1}}{{\sum^{n}_{j=1}}{{\beta}_{i,j}}(x_i \wedge y_j)}\\
&+{\sum^{n}_{j=1}}{{\gamma}_{j}}(y_j \wedge y_j)+{\sum^{2m}_{i=1}}{\sum_{\substack{j=1 \\ i\neq j}}^{n}{{\gamma}'_{i,j}}(y_i\wedge y_j)}.
\end{align*}

Let $\tilde{\kappa}: H_{(m,n)}\wedge H_{(m,n)} \to [H_{(m,n)},H_{(m,n)}]$ be given by $x\wedge y \mapsto [x,y]$. Since $\tilde{\kappa}(w)=0$, we have

\begin{align*}
w&={\sum^{2m}_{i=1}}{{\alpha}'_{i}}[x_i , x_{m+i}]+{\sum^{2m}_{i=1}}{\sum_{\substack{j=1 \\ j\neq m+i}}^{n}{{\alpha}_{i,j}}[x_i,x_j]}+{\sum^{2m}_{i=1}}{{\sum^{n}_{j=1}}{{\beta}_{i,j}}[x_i , y_j]}\\
&+{\sum^{n}_{j=1}}{{\gamma}_{j}}[y_j , y_j]+{\sum^{2m}_{i=1}}{\sum_{\substack{j=1 \\ i\neq J}}^{n}{{\gamma}'_{i,j}}[y_i, y_j]}=0.
\end{align*}

So $({\sum_{i=1}^{2m}}{{\alpha}'_{i}}+ {\sum^{n}_{j=1}}{{\gamma}_{j}})z=0$. Hence, 
${\sum^{2m}_{{i=1}}}{{\alpha}'_{i}}+ {\sum^{n}_{{j=1}}}{{\gamma}_{j}}=0$ and 
${\gamma}_n=-{\sum^{2m}_{{i=1}}}{{\alpha}'_{i}}- {\sum^{n-1}_{{j=1}}}{{\gamma}_{j}}$.
\\
Thus
\begin{align*}
w&={\sum^{2m}_{i=1}}{{\alpha}'_{i}}(x_i \wedge x_{m+i})+{\sum^{2m}_{i=1}}{\sum_{\substack{j=1 \\ j\neq m+i}}^{n}{{\alpha}_{i,j}}(x_i\wedge x_j)}+{\sum^{2m}_{i=1}}{\sum_{\substack{j=1 \\ i\neq j}}^{n}{{\beta}_{i,j}}(x_i\wedge y_j)}\\
&+{\sum^{n}_{j=1}}{{\gamma}_{j}}(y_j \wedge y_j)-({\sum^{2m}_{i=1}}{{\alpha}'_{i}}+{\sum^{n-1}_{j=1}}{{\gamma}_{j}})(y_m\wedge y_m).
\end{align*}

Therefore 

\begin{align*}
w&={\sum^{2m}_{i=1}}{{\alpha}'_{i}}(x_i \wedge x_{m+i}-y_m\wedge y_m)+{\sum^{2m}_{i=1}}{\sum_{\substack{j=1 \\ j\neq m+i}}^{n}{{\alpha}_{i,j}}(x_i\wedge x_j)}\\
&+{\sum^{2m}_{i=1}}{{\sum^{n}_{j=1}}{{\beta}_{i,j}}(x_i \wedge y_j)}+{\sum^{n-1}_{j=1}}{{\gamma}_{j}}(y_j \wedge y_j-y_m\wedge y_m).
\end{align*}

According to the definition $\mathcal{M}_0(H_{(m,n)})$, for all $i=1,...,2m\ , \ j=1,...,n$ such that $j\neq m+i$, ${x_i}\wedge {x_j}\in \mathcal{M}_0(H_{(m,n)})$ and for all $i=1,...,2m\ , \ j=1,...,n$ where $i\neq j$, ${x_i}\wedge {y_j}\in \mathcal{M}_0(H_{(m,n)})$.
\\
Also since 
\[{x_i}\wedge x_{m+i}-y_m\wedge y_m ={x_i}\wedge x_{m+i}+(-1)^{|y_m||y_m|}(y_m\wedge y_m),\] 
and
$$[x_i,x_{m+i}]+(-1)^{|y_m||y_m|}[y_m,y_m]=z-z=0.$$
Then $x_i\wedge x_{m+i}-y_m\wedge y_m\in \mathcal{M}_0(H_{(m,n)})$. Similarly, we can see that $(y_j\wedge y_j - y_m\wedge y_m) \in \mathcal{M}_0(H_{(m,n)})$. 
Thus $\mathcal{M}(H_{(m,n)})\subseteq \mathcal{M}_0(H_{(m,n)})$ and $\tilde{B_0}(H_{(m,n)})=0$.
\end{proof}
\end{theorem}

\begin{theorem}\label{t5.6}\cite{11}
Let $L$ be a nilpotent Lie superalgebra of dimension $(k|l)$ with $\dim L^2=(r|s)$, where $r+s=1$. Then we have 
\renewcommand {\labelenumi}{(\roman{enumi})}
\begin{enumerate}
\item{if $r=1$ , $s=0$ then $\tilde{B_0}(L)=0$,}
\item{if $r=0$ , $s=1$ then $\tilde{B_0}(L)=0$.}
\end{enumerate}
\begin{proof}
By Proposition $3.4$ in \cite{11}, if $r=1$ and $s=0$ then 
$$L\cong H_{(m,n)}\oplus A(k-2m-1\ |\ l-n) \ \  ;  \ \ m+n\geq 1.$$
Now using Proposition \ref{p4.8} and Theorem \ref{t5.5}, we have
\begin{align*}
\tilde{B_0}(L)&\cong \tilde{B_0}(H_{(m,n)}\oplus A(k-2m-1\ | \ l-n))\\
&\cong \tilde{B_0}(H_{(m,n)})\oplus \tilde{B_0}(A(k-2m-1 \ | \ l-n))=0.
\end{align*}
Since $\tilde{B_0}(H_{(m,n)})=\tilde{B_0}(A(k-2m-1 \ | \ l-n))=0$, the result follows. Also for $r=0$ and $s=1$, by using Proposition $3.4$ in \cite{11}, $L\cong H_m \oplus A(k-m \ | \ l-m-1)$. Similarly by using Proposition \ref{p4.8} and Theorem \ref{t5.5}, we have
\begin{align*}
\tilde{B_0}(L)&\cong \tilde{B_0}(H_{m}\oplus A(k-m\ | \ l-m-1))\\
&\cong \tilde{B_0}(H_{m})\oplus \tilde{B_0}(A(k-m \ | \ l-m-1))=0, 
\end{align*}
as required.
\end{proof}
\end{theorem}

One of the important results on the Schur multiplier of Lie superalgebras was presented by Narayan and et. al in \cite{13}. They showed that for a nilpotent Lie superalgebra $L$ of super dimension $(m|n)$, $\dim \mathcal{M}(L) = \frac{1}{2}[(m+n)^2+n-m]-t(L)$, for some $t(L)\geq 0$ that is called \textbf{corank}. Their results suggest an interesting question, ''can we classify Lie superalgebras of super dimension $(m|n)$ by corank?'' The answer to this question can be found for $t(L)=0,\ldots, 4$ in \cite{11}. Now, according to this classification, we will investigate the Bogomolov multiplier for some Lie superalgebras.

\begin{theorem}\label{t5.7}
Let $L$ be a $(m|n)$-super dimensional nilpotent Lie superalgebra and $t(L)\leq 3$. Then ${\tilde{B_0}}(L)=0$.
\begin{proof}
By Theorem $6.11$ in \cite{13}, Theorem $6.1$ in \cite{1}, Theorem \ref{t5.5} and Proposition \ref{p4.8}, ${\tilde{B_0}}(L)=0$.
\end{proof}
\end{theorem}

\begin{theorem}\label{t5.8}
Let $L$ be a $(m|n)$-super dimensional nilpotent Lie superalgebra and $t(L)=4$. Then ${\tilde{B_0}}(L)\neq 0$ if and only if 
$$L \cong <a,b,c,d,e\ \ | \ \ [a,b]=c , [a,c]=d , [a,d]=[b,c]=e>.$$
\begin{proof}
By using Theorem $6.11$ in \cite{13}, Theorems $6.1$ , $6.4$ in \cite{1}, Theorem \ref{t5.5} and Proposition \ref{p4.8}, the proof is obtained.
\end{proof}
\end{theorem}

In the following section, we want to classify all real Lie superalgebras $L$ of dimension at most $4$, such that $\tilde{B_0}(L)=0$. 

\section{Computing the Bogomolov multiplier of real Lie superalgebras of dimension at most $4$}\label{sec6}
This section is devoted to obtain the Bogomolov multiplier for the real Lie super algebras of dimension at most $4$ which are not Lie algebras. We need the classification of these Lie superalgebras in \cite{3}. Nigel Backhouse in \cite{3}, classified these Lie superalgebras into two types trivial and non trivial. Note that the Lie superalgebra $L$ is trivial, if $[L_{\bar{1}},L_{\bar{1}}]=0$, otherwise $L$ is non trivial. According to the all notations in \cite{3}, we also denote the elements of $L_{\bar{0}}$ $($resp $L_{\bar{1}})$ by Latin letters $($resp Greek letters$)$ taken from the begining of the alphabet.

\begin{theorem}\label{t6.1} 
Let $L_{(0,1)}$ be a trivial Lie superalgebra of dimension $1$ with a basis $\{\alpha\}$ and the only Lie super bracket $[\alpha,\alpha]=0$. Then ${\tilde{B_0}}(L_{(0,1)})=0$.
\begin{proof}
$L_{(0,1)}$ is abelian so, its Bogomolov multiplier is trivial.
\end{proof}
\end{theorem}

From \cite{3}, there is one trivial Lie superalgebra of dimension $2$ with a basis $\{a,\alpha\}$ and the only non-zero Lie super bracket $[a,\alpha]=\alpha$.
  
\begin{theorem}\label{t6.2}
Let $L_{(1,1)}$ be a trival Lie superalgebra of dimension $2$. Then \\
${\tilde{B_0}}(L_{(1,1)})=0$.
\begin{proof}
We can see that $L_{(1,1)}\wedge L_{(1,1)} = <a\wedge \alpha , \alpha\wedge \alpha>$. Hence, for all $w\in \mathcal{M}(L_{(1,1)})\subseteq L_{(1,1)}\wedge L_{(1,1)}$, there exist $\alpha_1 ,\alpha_2\in \mathbb{R}$, such that $w=\alpha_1(a\wedge \alpha)+\alpha_2(\alpha\wedge \alpha)$. Now, considering  $\tilde{\kappa} : L_{(1,1)}\wedge L_{(1,1)} \to [L_{(1,1)},L_{(1,1)}]$ given by $x\wedge y \to [x,y]$. Since $\tilde{\kappa} (w)=0$, we have ${{\alpha}_1}[a,\alpha]+{\alpha_2}[\alpha,\alpha]=0$. Thus, ${{\alpha}_1}{\alpha}=0$. So ${{\alpha}_1}=0$. Hence $w=\alpha_2(\alpha\wedge \alpha)$. Also we can see that
$\alpha\wedge \alpha=\alpha\wedge \alpha+(-1)^{|a||a|}a\wedge a$, such that $[\alpha,\alpha]+(-1)^{|a||a|}[a,a]=0$. Thus $\alpha\wedge \alpha \in \mathcal{M}_0(L_{(1,1)})$ and $\mathcal{M}(L_{(1,1)}) \subseteq \mathcal{M}_0(L_{(1,1)})$. Thus ${\tilde{B_0}}(L_{(1,1)})=0$.
\end{proof}
\end{theorem}

From \cite{3}, there are two types trivial Lie superalgebras of dimension $3$, which are denoted to $L_{(1,2)}$ and $L_{(2,1)}$. The Lie superalgebra $L_{(2,1)}$ has the basis ${a,b,\alpha}$, with the only non-zero Lie super brackets $[a,b]=b$, $[a,\alpha]=p{\alpha}$, where $p\neq
0$. Also we have following trivial Lie superalgebras of types $(1,2)$ that we showed them with $L^{i}_{(1,2)}$, where $i\in I=\{1,2,3,4\}$.
\begin{itemize}
\item{$L^{1}_{(1,2)}\cong <a,\alpha,\beta \ | \ [a,\alpha]=\alpha , [a,\beta]=p\beta \ \ ; \ \ 0<|p|\leq 1>,$}
\item{$L^{2}_{(1,2)}\cong <a,\alpha,\beta \ | \ [a,\beta]=\alpha>,$}
\item{$L^{3}_{(1,2)}\cong <a,\alpha,\beta \ | \ [a,\alpha]=\alpha , [a,\beta]=\alpha+\beta>,$}
\item{$L^{4}_{(1,2)}\cong <a,\alpha,\beta \ | \ [a,\alpha]=p\alpha-\beta , [a,\beta]=\alpha+p\beta \ \ ; \ \ p\geq 0>.$}
\end{itemize}

\begin{theorem}\label{t6.3}
Let $L$ be a trivial Lie superalgebra of dimension $3$. Then ${\tilde{B_0}}(L)=0$.
\begin{proof}
Let $L\cong L_{(2,1)}=<a,b,\alpha\  |\  [a,b]=b , [a,\alpha]=p{\alpha} \ ; \ p\neq 0>$. We have $L_{(2,1)}\wedge L_{(2,1)} = <a\wedge b , a\wedge \alpha , b\wedge \alpha , \alpha \wedge \alpha>$. So, for all $w\in \mathcal{M}(L_{(2,1)})\leq L_{(2,1)}\wedge L_{(2,1)}$, there exist $\alpha_1,\alpha_2,\alpha_3 ,\alpha_4 , \in \mathbb{R}$, such that $w=\alpha_1(a\wedge b)+\alpha_2(a\wedge \alpha)+\alpha_3(b\wedge \alpha)+\alpha_4(\alpha\wedge \alpha)$. Now, let $\tilde{\kappa} : L_{(2,1)}\wedge L_{(2,1)} \to {[L_{(2,1)},L_{(2,1)}]}$ given by $x\wedge y \to [x,y]$. Since $\tilde{\kappa} (w)=0$, we have $\alpha_1[a,b]+\alpha_2[a,\alpha]+\alpha_3[b,\alpha]+\alpha_4[\alpha,\alpha]=0$. So
 $\alpha_1 {b}+ \alpha_2 (p{\alpha})=0$ and so ${{\alpha}_1}= {{\alpha}_2}=0$. Thus, $w={{\alpha}_3}(b\wedge \alpha)+{{\alpha}_4}(\alpha \wedge \alpha)$. Also we can see that 
$$b\wedge \alpha=b\wedge \alpha+(-1)^{{|a|}{|a|}}a\wedge a \ \ \ , \ \ \ \alpha \wedge \alpha=\alpha \wedge \alpha+(-1)^{{|a|}{|a|}}a\wedge a.$$
So $w\in \mathcal{M}_0(L_{(2,1)})$ and $\mathcal{M}(L_{(2,1)}) \subseteq \mathcal{M}_0(L_{(2,1)})$. Hence ${\tilde{B_0}}(L_{(2,1)})=0$.
\\
Let $L\cong L^{4}_{(1,2)}=<a,\alpha,\beta \ | \ [a,\alpha]=p\alpha-\beta , [a,\beta]=\alpha+p\beta \ \ ; \ \ p\geq 0>$. We can check that $L^4_{(1,2)}\wedge L^4_{(1,2)} = <a\wedge \alpha , a\wedge \beta , \alpha\wedge \beta , \alpha \wedge \alpha , \beta\wedge \beta>$. Hence, for all $w\in \mathcal{M}(L^4_{(1,2)})\leq L^4_{(1,2)}\wedge L^4_{(1,2)}$, there exist $\alpha_1,\alpha_2,\alpha_3 ,\alpha_4 , \alpha_5\in \mathbb{R}$, such that $w=\alpha_1(a\wedge \alpha)+\alpha_2(a\wedge \beta)+\alpha_3(\alpha\wedge \beta)+\alpha_4(\alpha\wedge \alpha)+\alpha_5(\beta\wedge \beta)$. Now, considering  $\tilde{\kappa} : L^4_{(1,2)}\wedge L^4_{(1,2)} \to {[L^4_{(1,2)},L^4_{(1,2)}]}$ given by $x\wedge y \to [x,y]$. Since $\tilde{\kappa} (w)=0$, we have $\alpha_1[a,\alpha]+\alpha_2[a,\beta]+\alpha_3[\alpha,\beta]+\alpha_4[\alpha,\alpha]+\alpha_5[\beta,\beta]=0$. So
 $\alpha_1 (p\alpha-\beta) + \alpha_2 (\alpha+p{\beta})=0$ and 
$(\alpha_1{p}+\alpha_2)\alpha+(-\alpha_1+p{\alpha_2})\beta=0$.
 So, $\alpha_1{p}+\alpha_2=-\alpha_1+p{\alpha_2}=0$ and ${{\alpha}_1}= {{\alpha}_2}=0$. Thus, $w={{\alpha}_3}(\alpha\wedge \beta)+{{\alpha}_4}(\alpha \wedge \alpha)+{{\alpha}_5}(\beta\wedge \beta)$. \\ On the other hand, 
$$\alpha\wedge \beta=\alpha\wedge \beta+(-1)^{{|a|}{|a|}}a\wedge a \ \ \ , \ \ \ \alpha \wedge \alpha=\alpha \wedge \alpha+(-1)^{{|a|}{|a|}}a\wedge a$$
 $$\beta \wedge \beta=\beta \wedge \beta+(-1)^{{|a|}{|a|}}a\wedge a.$$
Thus $w\in \mathcal{M}_0(L^4_{(1,2)})$. So $\mathcal{M}(L^4_{(1,2)}) \subseteq \mathcal{M}_0(L^4_{(1,2)})$. Hence ${\tilde{B_0}}(L^4_{(1,2)})=0$. Similarly, we can see that ${\tilde{B_0}}(L^1_{(1,2)})={\tilde{B_0}}(L^2_{(1,2)})={\tilde{B_0}}(L^3_{(1,2)})=0$.

\end{proof}
\end{theorem}

From \cite{3}, there are three types trivial Lie superalgebras of dimension $4$, which are denoted to $L_{(3,1)}$, $L_{(2,2)}$ and $L_{(1,3)}$. We have the following presentations for trivial Lie superalgebras of types $(3,1)$, $(2,2)$ and $(1,3)$ that we denote them with $L^{i}_{(3,1)}$, $L^{i}_{(2,2)}$ and $L^{i}_{(1,3)}$ where $i\in I=\{1,...,6\}$.
\begin{itemize}
\item{$L^{1}_{(3,1)}\cong <a, b, c,\alpha \ | \ [b,c]=a , [b,\alpha]=\alpha>,$}
\item{$L^{2}_{(3,1)}\cong <a, b, c, \alpha \ | \ [a,c]=a , [b,c]=a+b , [c,\alpha]=q\alpha \ \ ; \ \ pq\neq 0>,$}
\item{$L^{3}_{(3,1)}\cong <a, b, c, \alpha, \ | \ [a,c]=pa-b , [b,c]=a+pb , [c,\alpha]=q\alpha \ \ ; \ \ q\neq 0>,$}
\item{$L^{1}_{(2,2)}\cong <a, b, \alpha, \beta \ | \ [a,\alpha]=\alpha , [a,\beta]=\beta , [b,\beta]=\alpha>,$}
\item{$L^{2}_{(2,2)}\cong <a, b, \alpha,\beta \ | \ [a,\alpha]=\alpha , [a,\beta]=\beta , [b,\alpha]=-\beta , [b,\beta]=\alpha>,$}
\item{$L^{3}_{(2,2)}\cong <a, b, \alpha,\beta \ | \ [a,b]=b , [a,\alpha]=p\alpha , [a,\beta]=q\beta \ \ ; \ \ pq\neq 0 \ , \ p\geq q>,$}
\item{$L^{4}_{(2,2)}\cong <a, b, \alpha,\beta \ | \ [a,b]=b , [a,\alpha]=p\alpha , [a,\beta]=\alpha+p\beta \ \ ; \ \ p\neq 0>,$}
\item{$L^{5}_{(2,2)}\cong <a, b, \alpha,\beta \ | \ [a,b]=b , [a,\alpha]=p\alpha-q\beta , [a,\beta]=q\alpha+p\beta \ \ ; \ \ q>0>,$}
\item{$L^{6}_{(2,2)}\cong <a, b, \alpha,\beta \ | \ [a,b]=b , [a,\alpha]=(p+1)\alpha , [a,\beta]=p\beta , [b,\beta]=\alpha>,$}
\item{$L^{1}_{(1,3)}\cong <a,\alpha,\beta, \gamma \ | \ [a,\alpha]=\alpha , [a,\beta]=p\beta , [a,\gamma]=q\gamma>,$}
\item{$L^{2}_{(1,3)}\cong <a,\alpha,\beta, \gamma \ | \ [a,\alpha]=\alpha , [a,\gamma]=\beta>,$}
\item{$L^{3}_{(1,3)}\cong <a,\alpha,\beta, \gamma \ | \ [a,\alpha]=p\alpha , [a,\beta]=\beta , [a,\gamma]=b+\gamma \ \ ; \ \ p\neq 0>,$}
\item{$L^{4}_{(1,3)}\cong <a,\alpha,\beta, \gamma \ | \ [a,\alpha]=p\alpha , [a,\beta]=q\beta-\gamma , [a,\gamma]=\beta+q\gamma \ \ ; \ \ q\geq 0 \ , \ p\neq 0>,$}
\item{$L^{5}_{(1,3)}\cong <a,\alpha,\beta, \gamma \ | \ [a,\beta]=\alpha , [a,\gamma]=\beta>,$}
\item{$L^{6}_{(1,3)}\cong <a,\alpha,\beta, \gamma \ | \ [a,\alpha]=\alpha , [a,\beta]=\alpha+\beta , [a,\gamma]=\beta+\gamma>.$}
\end{itemize}

\begin{theorem}\label{t6.4}
Let $L$ be a trivial Lie superalgebra of dimension $4$. Then ${\tilde{B_0}}(L)=0$.
\begin{proof}
Let $L\cong L^{3}_{(3,1)}=<a, b, c, \alpha, \ | \ [a,c]=pa-b , [b,c]=a+pb , [c,\alpha]=q\alpha \ \ ; \ \ q\neq 0>$. By using Definition \ref{d2}, we have
\begin{align*}
a\wedge \alpha &=([b,c]-pb)\wedge \alpha=[b,c]\wedge \alpha-p(b\wedge \alpha)\\
&=b\wedge [c,\alpha]-(-1)^{|b||c|}(c\wedge [b,\alpha])-p(b\wedge \alpha)\\
&=(q-p)(b\wedge \alpha).
\end{align*}
Thus, we have 
$$L^3_{(3,1)}\wedge L^3_{(3,1)} = <b\wedge \alpha , c\wedge \alpha , a\wedge b , a\wedge c , b\wedge c , \alpha \wedge \alpha>.$$ 
Hence, for all $w\in \mathcal{M}(L^3_{(3,1)})\subseteq L^3_{(3,1)}\wedge L^3_{(3,1)}$, there exist $\alpha_1,\alpha_2,\alpha_3 ,\alpha_4 , \alpha_5 , \alpha_6  \in \mathbb{R}$, such that $w=\alpha_1(b\wedge \alpha)+\alpha_2(c\wedge \alpha)+\alpha_3(a\wedge b)+\alpha_4(a\wedge c)+\alpha_5(b\wedge c)+\alpha_6(\alpha \wedge \alpha)$. Now, let $\tilde{\kappa} : L^3_{(3,1)}\wedge L^3_{(3,1)} \to {[L^3_{(3,1)},L^3_{(3,1)}]}$ given by $x\wedge y \to [x,y]$. Since $\tilde{\kappa} (w)=0$, we have $\alpha_1[b,\alpha]+\alpha_2[c,\alpha]+\alpha_3[a,b]+\alpha_4[a,c]+\alpha_5[b,c]+\alpha_6[\alpha,\alpha]=0$. So $\alpha_2(q\alpha)+\alpha_4(pa-b)+\alpha_5(a+pb)=0$ and $({\alpha_2}q)\alpha+({\alpha_4}p+\alpha_5)a+(-\alpha_4+{\alpha_5}p)b=0$. Since $q\neq 0$, $\alpha_2=0$ and $\alpha_4=\alpha_5=0$.
Thus, $w={{\alpha}_1}(b\wedge \alpha)+{{\alpha}_3}(a\wedge b)+{\alpha_6}(\alpha\wedge \alpha)$. Similar to the previous one, we can see that
$b\wedge \alpha , a\wedge b , \alpha\wedge \alpha \in \mathcal{M}_0(L^3_{(3,1)})$ and $w\in \mathcal{M}_0(L^3_{(3,1)})$. Thus $\mathcal{M}(L^3_{(3,1)}) \subseteq \mathcal{M}_0(L^3_{(3,1)})$. Hence ${\tilde{B_0}}(L^3_{(3,1)})=0$. Similarly, we have ${\tilde{B_0}}(L^1_{(3,1)})={\tilde{B_0}}(L^2_{(3,1)})=0$.
\\
\\
Let $L\cong L^6_{(2,2)}=<a, b, \alpha,\beta \ | \ [a,b]=b , [a,\alpha]=(p+1)\alpha , [a,\beta]=p\beta , [b,\beta]=\alpha>$. By using Definition \ref{d2}, we have
\begin{align*}
a\wedge \alpha &=a\wedge [b,\beta]=(-1)^{|\beta|(|a|+|b|)}([\beta,a]\wedge b)-(-1)^{|a||b|}([b,a]\wedge \beta) \\
&=-(-1)^{|\beta||a|}([a,\beta]\wedge b)+(-1)^{|b||a|}([a,b]\wedge \beta)\\
&=-p\beta \wedge b+b\wedge \beta=-p(\beta \wedge b)+b\wedge \beta\\
&=-p(-(-1)^{|\beta||b|}b\wedge \beta)+b\wedge \beta \\
&=(p+1)b\wedge \beta,
\end{align*}
and
\begin{align*}
\alpha \wedge \alpha &=\alpha\wedge [b,\beta]=(-1)^{|\beta|(|\alpha|+|b|)}([\beta,\alpha]\wedge b)-(-1)^{|\alpha||b|}([b,\alpha]\wedge \beta)=0
\end{align*}
Therefore, we have
$$L\cong L^6_{(2,2)}\wedge L\cong L^6_{(2,2)}=<a\wedge \beta , b\wedge \alpha , b\wedge \beta  , \alpha\wedge \beta , \beta\wedge \beta , a\wedge b>.$$
Hence, for all $w\in \mathcal{M}(L^6_{(2,2)})\subseteq L^6_{(2,2)}\wedge L^6_{(2,2)}$, there exist $\alpha_1,\alpha_2,\alpha_3 ,\alpha_4 , \alpha_5 , \alpha_6 \in \mathbb{R}$, such that 
\[w=\alpha_1(a\wedge \beta)+\alpha_2(b\wedge \alpha)+\alpha_3(b\wedge \beta)+\alpha_4(\alpha\wedge \beta)+\alpha_5(\beta\wedge \beta)+\alpha_6(a\wedge b).\]
Now, considering  $\tilde{\kappa} : L^6_{(2,2)}\wedge L^6_{(2,2)} \to {[L^6_{(2,2)},L^6_{(2,2)}]}$ given by $x\wedge y \to [x,y]$. Since $\tilde{\kappa} (w)=0$, we have $\alpha_1[a,\beta]+\alpha_2[b,\alpha]+\alpha_3[b,\beta]+\alpha_4[\alpha,\beta]+\alpha_5[\beta,\beta]+\alpha_6[a,b]=0$. So ${\alpha_1}(p\beta)+{\alpha_3}{\alpha}+{\alpha_6}b=0$. Thus if $p\neq 0$, then $\alpha_1=\alpha_3=\alpha_6=0$. Thus $w={\alpha_2}(b\wedge \alpha)+{\alpha_4}(\alpha \wedge \beta)+{\alpha_5}(\beta \wedge \beta)$ and $b\wedge \alpha , \alpha\wedge \beta , \beta\wedge \beta\in \mathcal{M}_0({L^6_{(2,2)}})$. Therefore $w\in \mathcal{M}_0(L^6_{(2,2)})$ and $\mathcal{M}(L^6_{(2,2)}) \subseteq \mathcal{M}_0(L^6_{(2,2)})$. Hence ${\tilde{B_0}}(L^6_{(2,2)})=0$. 
\\
But in case $p=0$, we have
$$L^6_{(2,2)}=<a, b, \alpha,\beta \ | \ [a,b]=b , [a,\alpha]=\alpha , [b,\beta]=\alpha>.$$
So 
\[w=\alpha_1(a\wedge \beta)+\alpha_2(b\wedge \alpha)+\alpha_4(\alpha\wedge \beta)+\alpha_5(\beta \wedge \beta).\]
Same as before, $a\wedge \beta , b\wedge \alpha , \alpha\wedge \beta , \beta\wedge \beta \in \mathcal{M}_0(L^6_{(2,2)})$ and finally, in this case, ${\tilde{B_0}}(L^6_{(2,2)})=0$. Similarly, we have ${\tilde{B_0}}(L^i_{(2,2)})=0$, where $i=1,...,5$.
\\
\\
Let 
$$L\cong L^4_{(1,3)}=<a,\alpha,\beta, \gamma \ | \ [a,\alpha]=p\alpha , [a,\beta]=q\beta-\gamma , [a,\gamma]=\beta+q\gamma \ \ ; \ \ q\geq 0 \ , \ p\neq 0>.$$
By using Definition \ref{d2}, we have
$$\alpha \wedge \beta=-p(\alpha\wedge \gamma) \ \ , \ \ \beta\wedge \beta=-q(\beta\wedge \gamma)+\gamma \wedge \gamma.$$
Therefore,
$$L^4_{(1,3)}\wedge L^4_{(1,3)}=<a\wedge \alpha , a\wedge \beta , a\wedge \gamma , \alpha\wedge \alpha , \alpha\wedge \gamma , \beta\wedge \gamma , \gamma\wedge \gamma>.$$
Hence, for all $w\in \mathcal{M}(L^4_{(1,3)})\subseteq L^4_{(1,3)}\wedge L^4_{(1,3)}$, there exist $\alpha_1,\alpha_2,\alpha_3 ,\alpha_4 , \alpha_5 , \alpha_6 , \alpha_7\in \mathbb{R}$, such that

\begin{align*}
w&=\alpha_1(a\wedge \alpha)+\alpha_2(a\wedge \beta)+\alpha_3(a\wedge \gamma)+\alpha_4(\alpha\wedge \alpha)\\
&+\alpha_5(\alpha\wedge \gamma)+\alpha_6(\beta\wedge \gamma)+\alpha_7(\gamma\wedge \gamma).
\end{align*}
Now, let $\tilde{\kappa} : L^4_{(1,3)}\wedge L^4_{(1,3)} \to {[L^4_{(1,3)},L^4_{(1,3)}]}$ given by $x\wedge y \to [x,y]$. Since $\tilde{\kappa} (w)=0$, we have \[\alpha_1[a,\alpha]+\alpha_2[a,\beta]+\alpha_3[a,\gamma]+\alpha_4[\alpha,\alpha]+\alpha_5[\alpha,\gamma]+\alpha_6[\beta,\gamma]+\alpha_7[\gamma,\gamma]=0.\]
So ${\alpha_1}(p\alpha)+{\alpha_2}(q\beta-\gamma)+{\alpha_3}(\beta+q\gamma)=0$. Hence $({\alpha_1}p)\alpha+({\alpha_2}q+\alpha_3)\beta+(-\alpha_2+{\alpha_3}q)\gamma=0$ and $\alpha_1=\alpha_2=\alpha_3=0$. Also we know 
\[(\alpha\wedge \alpha), (\alpha\wedge \gamma), (\beta\wedge \gamma), (\gamma\wedge \gamma)\in \mathcal{M}_0(L^4_{(1,3)}).\]
Therefore $\mathcal{M}(L^4_{(1,3)})\subseteq \mathcal{M}_0(L^4_{(1,3)})$ and ${\tilde{B_0}}(L^4_{(1,3)})=0$.
\end{proof}
\end{theorem}

\begin{corollary}
All trivial real Lie superalgebras of dimension at most $4$ have trivial Bogomolov multiplier.
\end{corollary}

Now we do the same calculations for all nontrivial real Lie superalgebras of dimension at most $4$.
\\
{\bf{5.1. Computing the Bogomolov multiplier of real nontrivial Lie superalgebras of dimension at most $4$}}
\\
From \cite{3}, there is one nontrivial Lie superalgebra of dimension $2$ with basis $a,\alpha$ and the only non-zero Lie super bracket $[\alpha,\alpha]=a$.
  
\begin{theorem}\label{t6.5}
Let $L$ be a nontrivial Lie superalgebra of dimension $2$. Then ${\tilde{B_0}}(L)=0$.
\begin{proof}
Let $L\cong L_{(1,1)}$. We can see that $L_{(1,1)}\wedge L_{(1,1)} = <a\wedge \alpha , \alpha\wedge \alpha>$. Hence, for all $w\in \mathcal{M}(L_{(1,1)})\subseteq L_{(1,1)}\wedge L_{(1,1)}$, there exist $\alpha_1 , \alpha_2\in \mathbb{R}$, such that $w=\alpha_1(a\wedge \alpha)+\alpha_2(\alpha\wedge \alpha)$. Now, considering  $\tilde{\kappa} : L_{(1,1)}\wedge L_{(1,1)} \to [L_{(1,1)},L_{(1,1)}]$ given by $x\wedge y \to [x,y]$. Since $\tilde{\kappa} (w)=0$, we have ${{\alpha}_1}[a,\alpha]+{\alpha_2}[\alpha,\alpha]=0$. Thus, ${{\alpha}_1}a=0$ and so ${{\alpha}_1}=0$. Hence $w=\alpha_2(\alpha\wedge \alpha)\in \mathcal{M}_0(L_{(1,1)})$ and $\mathcal{M}(L_{(1,1)}) \subseteq \mathcal{M}_0(L_{(1,1)})$. Thus ${\tilde{B_0}}(L_{(1,1)})=0$.
\end{proof}
\end{theorem}

From \cite{3}, there are two types nontrivial Lie superalgebras of dimension $3$, which are denoted by $L_{(1,2)}$ and $L_{(2,1)}$. The Lie superalgebra $L_{(2,1)}$ has the basis $\{a,b,\alpha \}$, with the only non-zero Lie super brackets $[a,b]=b$, $[a,\alpha]=1/2{\alpha}$ and $[\alpha,\alpha]=b$. Also we have the following nontrivial Lie superalgebras of types $(1,2)$ that we denote them by $L^{i}_{(1,2)}$, where $i\in I=\{1,2\}$.
\begin{itemize}
\item{$L^{1}_{(1,2)}\cong <a, \alpha , \beta \ | \ [\alpha,\alpha]=a , [\beta,\beta]=a>,$}
\item{$L^{2}_{(1,2)}\cong <a, \alpha, \beta \ | \ [\alpha,\alpha]=a , [\beta,\beta]=-a>.$}
\end{itemize}

\begin{theorem}\label{t6.6}
Let $L$ be a nontrivial Lie superalgebra of dimension $3$. Then ${\tilde{B_0}}(L)=0$.
\begin{proof}
Let $L\cong L^{1}_{(1,2)}=<a, \alpha , \beta \ | \ [\alpha,\alpha]=a , [\beta,\beta]=a>$. By using Definition \ref{d2}, $a\wedge \alpha, a\wedge \beta=0$. So we have
$L^1_{(1,2)}\wedge L^1_{(1,2)} = <\alpha\wedge \alpha , \alpha\wedge \beta , \beta\wedge \beta>$. Hence, for all $w\in \mathcal{M}(L^1_{(1,2)})\subseteq L^1_{(1,2)}\wedge L^1_{(1,2)}$, there exist $\alpha_1,\alpha_2,\alpha_3\in \mathbb{R}$, such that $w=\alpha_1(\alpha\wedge \alpha)+\alpha_2(\alpha\wedge \beta)+\alpha_3(\beta\wedge \beta)$. Now, let $\tilde{\kappa} : L^1_{(1,2)}\wedge L^1_{(1,2)} \to {[L^1_{(1,2)},L^1_{(1,2)}]}$ given by $x\wedge y \to [x,y]$. Since $\tilde{\kappa} (w)=0$, we have $\alpha_1[\alpha,\alpha]+\alpha_2[\alpha,\beta]+\alpha_3[\beta,\beta]=0$. So $(\alpha_1+\alpha_3)a=0$ and $\alpha_3=-\alpha_1$. Thus, $w={{\alpha}_1}(\alpha\wedge \alpha-\beta\wedge \beta)+{\alpha_2}(\alpha\wedge \beta)$. \\ 
Similar to the previous one, we can see that
$\alpha\wedge \beta \in \mathcal{M}_0(L^1_{(1,2)})$. On the other hand, we have
\[\alpha\wedge \alpha-\beta\wedge \beta=\alpha\wedge \alpha+(-1)^{|\beta||\beta|}\beta\wedge \beta \ \ \ ; \ \ \ [\alpha,\alpha]+(-1)^{|\beta||\beta|}[\beta,\beta]=0.\]
Therefore $\alpha\wedge \alpha+(-1)^{|\beta||\beta|}\beta\wedge \beta \in \mathcal{M}_0(L^1_{(1,2)})$ and $w\in \mathcal{M}_0(L^1_{(1,2)})$. Thus $\mathcal{M}(L^1_{(1,2)}) \subseteq \mathcal{M}_0(L^1_{(1,2)})$. Hence ${\tilde{B_0}}(L^1_{(1,2)})=0$. Similarly, ${\tilde{B_0}}(L^2_{(1,2)})=0$.
\end{proof}
\end{theorem}

From \cite{3}, there are three types nontrivial Lie superalgebras of dimension $4$, which are denoted by $L_{(3,1)}$, $L_{(2,2)}$ and $L_{(1,3)}$. We have the following presentations for nontrivial Lie superalgebras of types $(3,1)$, $(2,2)$ and $(1,3)$ that we denote them by $L^{i}_{(3,1)}$, $L^{i}_{(2,2)}$ and $L^{i}_{(1,3)}$ where $i\in I=\{1,...,17\}$.
\begin{itemize}
\item{$L^{1}_{(3,1)}\cong <a, b, c,\alpha \ | \ [b,c]=a , [\alpha,\alpha]=a>,$}
\item{$L^{2}_{(3,1)}\cong <a, b, c, \alpha \ | \ [a,b]=b , [a,c]=pc , [a,\alpha]=\frac{1}{2}{\alpha} , [\alpha,\alpha]=b\ \ ; \ \ p\neq 0>,$}
\item{$L^{3}_{(3,1)}\cong <a, b, c, \alpha, \ | \ [a,b]=b , [a,c]=b+c , [a,\alpha]=\frac{1}{2}{\alpha} , [\alpha,\alpha]=b>,$}
\item{$L^{4}_{(3,1)}\cong <a, b, c, \alpha, \ | \ [a,b]=b , [a,c]=-b+c , [a,\alpha]=\frac{1}{2}{\alpha} , [\alpha,\alpha]=b>,$}
\item{$L^{1}_{(2,2)}\cong <a, b, \alpha, \beta \ | \ [a,b]=b , [a,\alpha]=\frac{1}{2}{\alpha} , [a,\beta]=\frac{1}{2}{\beta} , [\alpha,\alpha]=b , [\beta,\beta]=b>,$}
\item{$L^{2}_{(2,2)}\cong <a, b, \alpha,\beta \ | \ [a,b]=b , [a,\alpha]=\frac{1}{2}{\alpha} , [a,\beta]=\frac{1}{2}{\beta} , [\alpha,\alpha]=b , [\beta,\beta]=-b>,$}
\item{$L^{3}_{(2,2)}\cong <a, b, \alpha,\beta \ | \ [a,b]=b , [a,\alpha]=\frac{1}{2}{\alpha} , [a,\beta]=\frac{1}{2}{\beta} , [\alpha,\alpha]=b>,$}
\item{$L^{4}_{(2,2)}\cong <a, b, \alpha,\beta \ | \ [a,b]=b , [a,\alpha]=p{\alpha} , [a,\beta]=(1-p){\beta} , [\alpha,\beta]=b \ \ ; \ \ p\leq \frac{1}{2}>,$}
\item{$L^{5}_{(2,2)}\cong <a, b, \alpha,\beta \ | \ [a,b]=b , [a,\alpha]=\frac{1}{2}{\alpha} , [a,\beta]=\alpha+\frac{1}{2}{\beta} , [\beta,\beta]=b>,$}
\item{$L^{6}_{(2,2)}\cong <a, b, \alpha,\beta \ | \ [a,b]=b , [a,\alpha]=\frac{1}{2}{\alpha}-p\beta , [a,\beta]=p\alpha+\frac{1}{2}{\beta} , [\alpha,\alpha]=b , [\beta,\beta]=b \ \ ; \ \ p>0>,$}
\item{$L^{7}_{(2,2)}\cong <a, b, \alpha,\beta \ | \ [a,b]=b , [a,\alpha]={\alpha} , [b,\beta]=\alpha , [\beta,\beta]=a , [\alpha,\beta]=-\frac{1}{2}b>,$}
\item{$L^{8}_{(2,2)}\cong <a, b, \alpha,\beta \ | \ [a,b]=b , [a,\alpha]={\alpha} , [b,\beta]=\alpha , [\beta,\beta]=-a , [\alpha,\beta]=\frac{1}{2}b>,$}
\item{$L^{9}_{(2,2)}\cong <a, b, \alpha,\beta \ | \ [\alpha,\alpha]=a , [\beta,\beta]=b>,$}
\item{$L^{10}_{(2,2)}\cong <a, b, \alpha,\beta \ | \ [\alpha,\alpha]=a , [\beta,\beta]=b , [\alpha,\beta]=a>,$}
\item{$L^{11}_{(2,2)}\cong <a, b, \alpha,\beta \ | \ [\alpha,\alpha]=a , [\beta,\beta]=b , [\alpha,\beta]=p(a+b) \ \ ; \ \ p>0>,$}
\item{$L^{12}_{(2,2)}\cong <a, b, \alpha,\beta \ | \ [\alpha,\alpha]=a , [\beta,\beta]=b , [\alpha,\beta]=p(a-b) \ \ ; \ \ p>0>,$}
\item{$L^{13}_{(2,2)}\cong <a, b, \alpha,\beta \ | \ [a,b]=b , [a,\alpha]=\alpha , [\alpha,\beta]=b>,$}
\item{$L^{14}_{(2,2)}\cong <a, b, \alpha,\beta \ | \ [a,b]=b , [a,\alpha]=\frac{1}{2}\alpha , [\alpha,\alpha]=b>,$}
\item{$L^{15}_{(2,2)}\cong <a, b, \alpha,\beta \ | \ [a,\alpha]=\alpha , [a,\beta]=-\beta , [a,\beta]=b>,$}
\item{$L^{16}_{(2,2)}\cong <a, b, \alpha,\beta \ | \ [a,\beta]=\alpha , [\beta,\beta]=b>,$}
\item{$L^{17}_{(2,2)}\cong <a, b, \alpha,\beta \ | \ [a,\alpha]=-\beta , [a,\beta]=\alpha , [\alpha,\alpha]=b , [\beta,\beta]=b>,$}
\item{$L^{1}_{(1,3)}\cong <a,\alpha,\beta, \gamma \ | \ [\alpha,\alpha]=a , [\beta,\beta]=a , [\gamma,\gamma]=a>,$}
\item{$L^{2}_{(1,3)}\cong <a,\alpha,\beta, \gamma \ | \ [\alpha,\alpha]=a , [\beta,\beta]=a , [\gamma,\gamma]=-a>.$}

\end{itemize}

\begin{theorem}\label{t6.7}
Let $L$ be a nontrival Lie superalgebra of dimension $4$. Then\\ ${\tilde{B_0}}(L)=0$.
\begin{proof}
Let $L\cong L^{3}_{(3,1)}=<a, b, c, \alpha, \ | \ [a,b]=b , [a,c]=b+c , [a,\alpha]=\frac{1}{2}{\alpha} , [\alpha,\alpha]=b>$. According to the Definition \ref{d2}
$$a\wedge b=\alpha \wedge \alpha \ \ , \ \ b\wedge c=0 \ \ , \ \ c\wedge \alpha=-\frac{2}{3}(b\wedge \alpha),$$ 
We have $L^3_{(3,1)}\wedge L^3_{(3,1)} = <a\wedge b , a\wedge c , a\wedge \alpha , b\wedge \alpha>$. Hence, for all $w\in \mathcal{M}(L^3_{(3,1)})\subseteq L^3_{(3,1)}\wedge L^3_{(3,1)}$, there exist $\alpha_1,\alpha_2,\alpha_3 ,\alpha_4  \in \mathbb{R}$, such that $w=\alpha_1(a\wedge b)+\alpha_2(a\wedge c)+\alpha_3(a\wedge \alpha)+\alpha_4(b\wedge \alpha)$. Now, let $\tilde{\kappa} : L^3_{(3,1)}\wedge L^3_{(3,1)} \to {[L^3_{(3,1)},L^3_{(3,1)}]}$ be given by $x\wedge y \to [x,y]$. Since $\tilde{\kappa} (w)=0$, we have $\alpha_1[a,b]+\alpha_2[a,c]+\alpha_3[a,\alpha]+\alpha_4[b,\alpha]=0$ and so $(\alpha_1+\alpha_2)b+{\alpha_2}c+\frac{1}{2}{\alpha_3}{\alpha} =0$ and $\alpha_1=\alpha_2=\alpha_3=0$. Thus, $w=\alpha_4(b\wedge \alpha)$ and  $b\wedge \alpha \in \mathcal{M}_0(L^3_{(3,1)})$. Hence $\mathcal{M}(L^3_{(3,1)}) \subseteq \mathcal{M}_0(L^3_{(3,1)})$. Thus ${\tilde{B_0}}(L^3_{(3,1)})=0$. 
\\
\\
Now, let $L\cong L^1_{(2,2)}=<a, b, \alpha, \beta \ | \ [a,b]=b , [a,\alpha]=\frac{1}{2}{\alpha} , [a,\beta]=\frac{1}{2}{\beta} , [\alpha,\alpha]=b , [\beta,\beta]=b>$. Since 
$$a\wedge b=\alpha \wedge \alpha=\beta \wedge \beta \ \ , \ \ b\wedge \alpha=b\wedge \beta=0,$$
we have
$$L^1_{(2,2)}\wedge L^1_{(2,2)}=<a\wedge b , a\wedge \alpha , a\wedge \beta , \alpha\wedge \beta>.$$
For all $w\in \mathcal{M}(L^1_{(2,2)})\subseteq L^1_{(2,2)}\wedge L^1_{(2,2)}$, there exist $\alpha_1,\alpha_2,\alpha_3,\alpha_4\in \mathbb{R}$ such that $$w=\alpha_1(a\wedge b)+\alpha_2(a\wedge \alpha)+\alpha_3(a\wedge \beta)+\alpha_4(\alpha\wedge \beta).$$
Now let $\tilde{\kappa}: L^1_{(2,2)}\wedge L^1_{(2,2)}\to [L^1_{(2,2)},L^1_{(2,2)}]$ given by $x\wedge y\to [x,y]$. Since $\tilde{\kappa}(w)=0$, we have
$$\alpha_1[a,b]+\alpha_2[a,\alpha]+\alpha_3[a,\beta]+\alpha_4[\alpha,\beta]=0.$$
Thus ${\alpha_1}b+\frac{1}{2}{\alpha_2}\alpha+\frac{1}{2}{\alpha_3}\beta=0$, so $\alpha_1=\alpha_2=\alpha_3=0$. Hence $w=\alpha_4(\alpha\wedge \beta)$. So, $w\in \mathcal{M}_0(L^1_{(2,2)})$ and $\mathcal{M}(L^1_{(2,2)})\subseteq \mathcal{M}_0(L^1_{(2,2)})$. Therefore ${\tilde{B_0}}(L^1_{(2,2)})=0$. 
\\
\\

Let $L\cong L^{11}_{(2,2)}=<a,b,\alpha,\beta \ | \ [\alpha,\alpha]=a , [\beta,\beta]=a , [\alpha,\beta]=p(a+b) \ ; \ p>0>$. \\By using Definition \ref{d2}, we have
$$a\wedge b=[\alpha,\alpha]\wedge b=\alpha\wedge [\alpha,b]-(-1)^{|\alpha||\alpha|}(\alpha\wedge [\alpha,b])=0,$$
and
$$a\wedge \beta=2p(\alpha\wedge a+\alpha\wedge b) \ \ , \ \ b\wedge \alpha=2p(\beta\wedge a +\beta\wedge b).$$
Hence, we see that
$$L^{11}_{(2,2)}\wedge L^{11}_{(2,2)}=<a\wedge \alpha, b\wedge \alpha, \alpha\wedge \alpha, \alpha\wedge \beta, \beta\wedge \beta>.$$
So for all $w\in \mathcal{M}(L^{11}_{(2,2)})\subseteq L^{11}_{(2,2)}\wedge L^{11}_{(2,2)}$, there exist $\alpha_1,\alpha_2,\alpha_3,\alpha_4,\alpha_5 \in \mathbb{R}$ such that
\begin{align*}
w&=\alpha_1(a\wedge \alpha)+\alpha_2(b\wedge \alpha)+\alpha_3(\alpha\wedge \alpha)\\
&+\alpha_4(\alpha\wedge \beta)+\alpha_5(\beta\wedge \beta).
\end{align*}
By using $\tilde{\kappa}: L^{11}_{(2,2)}\wedge L^{11}_{(2,2)}\to [L^{11}_{(2,2)},L^{11}_{(2,2)}]$ given by $x\wedge y\to [x,y]$. Since $\tilde{\kappa}(w)=0$, we have
$$\alpha_1[a,\alpha]+\alpha_2[b,\alpha]+\alpha_3[\alpha,\alpha]+\alpha_4[\alpha,\beta]+\alpha_5[\beta,\beta]=0.$$
Thus $(\alpha_3+p\alpha_4)a+(\alpha_5+p\alpha_4)b=0$ and so,  $\alpha_3=\alpha_5=-p\alpha_4$. Hence
\begin{align*}
w&=\alpha_1(a\wedge \alpha)+\alpha_2(b\wedge \alpha)+\alpha_4(\alpha\wedge \beta-p\alpha\wedge \alpha-p\beta\wedge \beta).
\end{align*}
On the other hand, we have
$$\alpha\wedge \beta-p\alpha\wedge \alpha-p\beta\wedge \beta=((\alpha-p\beta)\wedge \beta)-(p\alpha\wedge \alpha)=((\alpha-p\beta)\wedge \beta)+(-1)^{|p\alpha||\alpha|}(p\alpha\wedge \alpha),$$
and
$$[\alpha-p\beta,\beta]+(-1)^{|p\alpha||\alpha|}[p\alpha,\alpha]=0.$$
Thus $\alpha\wedge \beta-p\alpha\wedge \alpha-p\beta\wedge \beta\in \mathcal{M}_0(L^{11}_{(2,2)})$. Also, $a\wedge \alpha, b\wedge \alpha \in \mathcal{M}_0(L^{11}_{(2,2)})$. Therefore, $w\in \mathcal{M}_0(L^{11}_{(2,2)})$. So $\mathcal{M}(L^{11}_{(2,2)})\subseteq \mathcal{M}_0(L^{11}_{(2,2)})$. Hence $\tilde{B_0}(L^{11}_{(2,2)})=0$. Similarly, $\tilde{B_0}(L^{12}_{(2,2)})=0$.
\\
\\
Let $L\cong L^1_{(1,3)}=<a,\alpha,\beta, \gamma \ | \ [\alpha,\alpha]=a , [\beta,\beta]=a , [\gamma,\gamma]=a>$. 
According to the Definition \ref{d2}, we have
$$a\wedge \alpha = a\wedge \beta = a\wedge \gamma =0.$$
Hence, we see that
$$L^1_{(1,3)}\wedge L^1_{(1,3)}=<\alpha\wedge \alpha , \alpha\wedge \beta , \alpha\wedge \gamma , \beta\wedge \beta , \beta\wedge \gamma , \gamma\wedge \gamma>.$$
So for all $w\in \mathcal{M}(L^1_{(1,3)})\subseteq L^1_{(1,3)}\wedge L^1_{(1,3)}$, there exist $\alpha_1,\alpha_2,\alpha_3,\alpha_4,\alpha_5,\alpha_6 \in \mathbb{R}$ such that
\begin{align*}
w&=\alpha_1(\alpha\wedge \alpha)+\alpha_2(\alpha\wedge \beta)+\alpha_3(\alpha\wedge \gamma)\\
&+\alpha_4(\beta\wedge \beta)+\alpha_5(\beta\wedge \gamma)+\alpha_6(\gamma\wedge \gamma).
\end{align*}
By using a Lie super homomorphism $\tilde{\kappa}: L^1_{(1,3)}\wedge L^1_{(1,3)}\to [L^1_{(1,3)},L^1_{(1,3)}]$ given by $x\wedge y\to [x,y]$, since $\tilde{\kappa}(w)=0$, we have
$$\alpha_1[\alpha,\alpha]+\alpha_2[\alpha,\beta]+\alpha_3[\alpha,\gamma]+\alpha_4[\beta,\beta]+\alpha_5[\beta,\gamma]+\alpha_6[\gamma,\gamma]=0.$$
Thus $(\alpha_1+\alpha_4+\alpha_6)a=0$ and $\alpha_6=-\alpha_1-\alpha_4$. Hence
\begin{align*}
w&=\alpha_1(\alpha\wedge \alpha-\gamma\wedge \gamma)
+\alpha_2(\alpha\wedge \beta)+\alpha_3(\alpha\wedge \gamma)\\
&+\alpha_4(\beta\wedge \beta-\gamma\wedge \gamma)+\alpha_5(\beta\wedge \gamma).
\end{align*}
Now, since 
$$\alpha \wedge \alpha - \gamma \wedge \gamma=\alpha \wedge \alpha+(-1)^{|\gamma||\gamma|}(\gamma \wedge \gamma) \ \ , \ \ [\alpha,\alpha]+(-1)^{|\gamma||\gamma|}[\gamma,\gamma]=0,$$
and
$$\beta \wedge \beta - \gamma \wedge \gamma=\beta \wedge \beta+(-1)^{|\gamma||\gamma|}(\gamma \wedge \gamma) \ \ , \ \ [\beta,\beta]+(-1)^{|\gamma||\gamma|}[\gamma,\gamma]=0,$$
 we have $(\alpha \wedge \alpha - \gamma \wedge \gamma), (\beta \wedge \beta - \gamma \wedge \gamma)\in \mathcal{M}_0(L^1_{(1,3)})$. 
 \\
 Also, $\alpha\wedge \beta, \alpha\wedge \gamma, \beta \wedge \gamma \in \mathcal{M}_0(L^1_{(1,3)})$. Thus, $w\in \mathcal{M}_0(L^1_{(1,3)})$. So $\mathcal{M}(L^1_{(1,3)})\subseteq \mathcal{M}_0(L^1_{(1,3)})$. Hence $\tilde{B_0}(L^1_{(1,3)})=0$. Similarly, $\tilde{B_0}(L^2_{(1,3)})=0$.
 \\
In general, for $i=1,2,4$, $\tilde{B_0}(L^i_{(3,1)})=0$, for $i=2,...,17$, $\tilde{B_0}(L^i_{(2,2)})=0$ and $\tilde{B_0}(L^2_{(1,3)})=0$.

\end{proof}
\end{theorem}

\begin{corollary}
All nontrivial real Lie superalgebras of dimension at most $4$ have trivial Bogomolov multiplier.
\end{corollary}

\end{document}